\documentclass[12pt,reqno]{amsart}
\usepackage{amsmath, amssymb, amsthm, amsfonts}

\usepackage{mathrsfs} \usepackage{upgreek}
\usepackage{amscd}
\usepackage{graphicx}
\usepackage[all]{xy}
\usepackage{cite}
\usepackage{tikz}
\usepackage{color} \usepackage{tikz-cd}
\usepackage{hyperref}
\usepackage{comment}
\usepackage[english]{babel}
\usepackage[margin=1.0in]{geometry}
\usepackage[a-1b]{pdfx}

\newcommand{\CC}{\mathbb C}
\newcommand{\EE}{\mathbb E}
\newcommand{\PP}{\mathbb P}

\newcommand{\vir}{\mathrm{vir}}

\DeclareMathOperator{\Aut}{Aut}
\DeclareMathOperator{\Hom}{Hom}
\DeclareMathOperator{\Cont}{Cont}
\DeclareMathOperator{\ch}{ch}
\DeclareMathOperator{\ev}{ev}
\DeclareMathOperator{\St}{St}
\theoremstyle{plain}
\newtheorem{theorem}{Theorem}
\newtheorem{lemma}[theorem]{Lemma}
\newtheorem{proposition}[theorem]{Proposition}
\newtheorem{corollary}[theorem]{Corollary}
\newtheorem{example}[theorem]{Example}

\newtheorem{conjecture}[theorem]{Conjecture}

\theoremstyle{definition}
\newtheorem{definition}[theorem]{Definition}
\newtheorem{remark}[theorem]{Remark}
\newtheorem{question}[theorem]{Question}

\usepackage{color}

\begin{document}

\title[Universal equations from Hodge integrals]{Universal equations for higher genus Gromov--Witten invariants from Hodge integrals}
\author[F.~Janda]{Felix Janda}
\address{University of Illinois Urbana--Champaign, Urbana, IL 61801, USA}
\email{fjanda@illinois.edu}

\author[X.~Wang]{Xin Wang}
\address{School of Mathematics, Shandong University, Jinan, China}
\email{wangxin2015@sdu.edu.cn}

\begin{abstract}
  We establish new universal equations for higher genus Gromov--Witten
  invariants of target manifolds, by studying both the Chern character
  and Chern classes of the Hodge bundle on the moduli space of curves.
  As a consequence, we find new push-forward relations on the moduli
  space of stable curves.
\end{abstract}


\maketitle	
\tableofcontents
\allowdisplaybreaks

\section{Introduction}

Finding universal partial differential equations for Gromov--Witten invariants of any target manifold is a long-standing problem in Gromov--Witten theory.
Such equations capture fundamental properties for Gromov--Witten invariants.
Usually, it is very difficult to find explicit equations in all genera.
So far, there have been mainly two approaches for finding such universal equations for Gromov--Witten invariants.
One approach is to use tautological relations, in particular \emph{topological recursion relations}, on the moduli space of stable curves.
For example, the equality between the boundary divisors on $\overline{\mathcal{M}}_{0,4}$ leads to the well-known WDVV equation for genus-0 Gromov--Witten invariants.
Another approach is to use tautological relations on the moduli space of stable maps.
For example, in \cite{faber2000hodge}, by using the Grothedieck--Riemann--Roch formula on the moduli space of stable maps, Faber and Pandharipande obtained a family of explicit universal equations for descendant Gromov--Witten invariants.
In 2014, a remarkable formula for double ramification cycles was proposed by Pixton and later proved in \cite{janda2017double}.
In this paper, we use this new ingredient to obtain new universal equations for descendant Gromov--Witten invariants.

\subsection{New universal equations}

Let $X$ be a smooth projective variety and $\{\phi_\alpha: \alpha=1,\dots,N\}$ be a basis of its cohomology ring $H^*(X;\mathbb{C})$ with $\phi_1=\mathbf{1}$ the identity.
Recall that the \emph{big phase space} for Gromov--Witten invariants of $X$ is defined to be $\prod_{n=0}^{\infty}H^*(X;\mathbb{C})$ with standard
basis $\{\tau_{n}(\phi_\alpha): \alpha=1,\dots,N, n\geq0\}$.
Denote the coordinates on the big phase space
with respect to the standard basis by  $\{t_n^\alpha\}$.
Let $F_g$ be the genus-$g$ generating function, which is a formal power
series of $\mathbf{t}=(t_n^\alpha)$  with coefficients being the  genus-$g$ Gromov-Witten invariants.
Denote $\langle\langle\tau_{n_1}(\phi_{\alpha_1}),\dots,\tau_{n_k}(\phi_{\alpha_k})\rangle\rangle_g$ for the derivatives
of $F_g$ with respect to the variables $t_{n_1}^{\alpha_1},\dots,t_{n_k}^{\alpha_k}$.
Similarly, we use $\langle\langle\tau_{n_1}(\phi_{\alpha_1}),\dots,\tau_{n_k}(\phi_{\alpha_k});c_{g,k}\rangle\rangle_g$ to denote  the correlation function of Gromov--Witten theory twisted with a cohomology class $c_{g,k} \in H^*(\overline{\mathcal{M}}_{g,k};\mathbb{C})$.
For convenience, we identify $\tau_n(\phi_\alpha)$ with
the coordinate vector field $\frac{\partial}{\partial t_n^\alpha}$
on $\prod_{n=0}^{\infty}H^*(X;\mathbb{C})$ for $n\geq0$. If $n<0$, $\tau_n(\phi_\alpha)$ is understood to be the $0$
vector field. We also abbreviate $\tau_0(\phi_\alpha)$ by $\phi_\alpha$. Let $\phi^\alpha=\eta^{\alpha\beta}\phi_\beta$  with $(\eta^{\alpha\beta})$
 representing the inverse matrix of the Poincar\'e intersection pairing on $H^*(X;\mathbb{C})$. As a  convention, 
repeated Greek letter indices are summed over their entire range.
Recall 
 the following $T$ operator which was studied in
\cite{liu2002quantum}
\begin{equation}
  \label{eq:T-operator}
  T(W):=\tau_+(W)-\langle\langle W\phi^\alpha\rangle\rangle_0\phi_\alpha
\end{equation}
for any vector field $W$ on the big phase space, where $\tau_+(W)$  is a linear operator defined
by $\tau_{+}(\tau_n(\phi_\alpha))=\tau_{n+1}(\phi_\alpha)$.

Let $A=(a_1,\dots,a_n)$ be a vector with integers $\sum_{i=1}^{n}a_i=0$.  Let $\mathsf{P}_g^g(A)\in H^{2g}(\overline{\mathcal{M}}_{g,n})$ be the cohomology class obtained from  Pixton's double ramification cycle formula\cite{janda2017double}, which is roughly a linear combination of descendant stratum classes.

Our first result is:
\begin{theorem}\label{thm:deg-g-chern-chara} For any $g\geq2$, the following equation holds for pure descendant Gromov-Witten invariants of $X$
\begin{small}
\begin{align}\label{eqn:deg-g-chern-chara}
&-\sum_{n,\alpha}\Tilde{t}_n^\alpha\langle\langle\tau_{n+2g-1}(\phi_\alpha)\rangle\rangle_g+\frac{1}{2}\sum_{i=0}^{2g-2}(-1)^{i}
\left\{\langle\langle\tau_i(\phi_\alpha)\tau_{2g-2-i}(\phi^\alpha)\rangle\rangle_{g-1}
+\sum_{h=0}^{g}\langle\langle\tau_{i}(\phi_{\alpha})\rangle\rangle_h
\langle\langle\tau_{2g-2-i}(\phi^{\alpha})\rangle\rangle_{g-h}\right\}
\nonumber
\\=&\frac{(-1)^{g}g}{2^{2g-1}}
\sum_{m = 1}^{g-1} \frac{1}{m}\sum_{\substack{k_1+\dots+k_m=g\\ k_1,\dots,k_m\geq1}} \sum_{\substack{g_1+\dots+g_m=g-1\\ g_1,\dots,g_m\geq1}}\sum_{\alpha_{1},\dots,\alpha_{m}=1}^{N}
\prod_{i=1}^{m}\frac{-1}{k_i!}\sum_{l_i=0}^{k_i-1}\binom{k_i-1}{l_i}
\langle\langle T^{l_i}(\phi_{\alpha_{i-1}})T^{k_i-1-l_i}(\phi^{\alpha_i});\mathsf{P}_{g_i}^{g_i}(0,0)\rangle\rangle_{g_i}
\end{align}    
\end{small}
where $\tilde{t}_n^\alpha=t_n^\alpha-\delta_{n}^{1}\delta_\alpha^{1}$ and $\alpha_0=\alpha_m$. 
% The equation~\eqref{eqn:deg-g-chern-chara} also holds for $g=1$ if we set its right hand side to be $\frac{1}{2}\sum_{\alpha}\langle\langle\phi_\alpha\phi^\alpha\rangle\rangle_0$. 
\end{theorem}
% Equation \eqref{eqn:deg-g-chern-chara} can be seen as a generalization of Faber--Pandharipande equations in \cite{faber2000hodge}.
Together with \cite[Theorem 4]{liu2011new}, Theorem~\ref{thm:deg-g-chern-chara} implies the following identity:
\begin{corollary}
  \label{cor:LP-push-g}
  For any $g\geq2$, the following equation holds for the pure
  descendant Gromov-Witten invariants of $X$
  \begin{small}
\begin{align}\label{eqn:LP-push-g}
&\sum_{i=0}^{2g-2}(-1)^{i}
\langle\langle T^i(\phi_\alpha)T^{2g-2-i}(\phi^\alpha)\rangle\rangle_{g-1}
\nonumber
\\=&\frac{(-1)^{g}g}{2^{2g-2}}
\sum_{m = 1}^{g-1} \frac{1}{m}\sum_{\substack{k_1+\dots+k_m=g\\ k_1,\dots,k_m\geq1}} \sum_{\substack{g_1+\dots+g_m=g-1\\ g_1,\dots,g_m\geq1}}
\sum_{\alpha_1,\dots,\alpha_m=1}^{N}\prod_{i=1}^{m}\frac{-1}{ k_i!}\sum_{l_i=0}^{k_i-1}\binom{k_i-1}{l_i}
\langle\langle T^{l_i}(\phi_{\alpha_{i-1}})T^{k_i-1-l_i}(\phi^{\alpha_i});\mathsf{P}_{g_i}^{g_i}(0,0)\rangle\rangle_{g_i}.
\end{align}  \end{small}
\end{corollary}
Proposition~\ref{prop:TaTb} in Section~\ref{subsec:Reformulation of Theorem 3} explains how to view the right hand side in \eqref{eqn:deg-g-chern-chara} and \eqref{eqn:LP-push-g} in terms of pure descendant Gromov-Witten invariants.
\begin{remark}
  In \cite{faber2000hodge}, it was pointed out that
  Faber--Pandharipande's equations are related to the genus-0
  $\widetilde{L}_n$-constraints proposed by Eguchi-Hori-Xiong in
  \cite{eguchi1997quantum}.
  However, the explicit expression of the higher genus
  $\widetilde{L}_n$-constraints remain unknown.
  We hope that \eqref{eqn:deg-g-chern-chara} will give a hint at the
  explicit form of the higher genus corrections to the
  $\widetilde{L}_n$-constraints.
\end{remark}
\subsection{A new push-forward relation}
It  is very interesting to find   nontrivial classes in the kernel
of the boundary push-forward map
\[\iota\colon R^*(\overline{\mathcal{M}}_{g,2})\rightarrow R^*(\overline{\mathcal{M}}_{g+1})\]
which can also be written in terms of $\psi$ classes and boundary classes.
Via the splitting principle in genus $g+1$ Gromov-Witten theory, this yields universal equations
in genus $g$ from linear combinations of descendant boundary classes in the kernel of $\iota_*$.
Such tautological relations were first constructed by Liu--Pandharipande in $R^{2g+r+1}(\overline{\mathcal{M}}_{g+1})$ for $r\geq1$ (cf. \cite[Theorem 2]{liu2011new}).  
It is remarkable that the equation
in Corollary~\ref{cor:LP-push-g} can be lifted to cycle level, which gives an extension  of Liu--Pandharipande's relation to the case of $r=0$.
\tikz{\coordinate (A) at (0,0); \coordinate (B) at (1,0); \coordinate (C) at (0.6,0.5);}

\tikzset{baseline=0, label distance=-3mm}
\def\NC{\draw (0,0.25) circle(0.25);}
\def\NL{\draw plot [smooth,tension=1.5] coordinates {(0,0) (-0.2,0.5) (-0.5,0.2) (0,0)};}
\def\NR{\draw plot [smooth,tension=1.5] coordinates {(0,0) (0.2,0.5) (0.5,0.2) (0,0)};}
\def\NN{\NL\NR}
\def\NNN{\NN \begin{scope}[rotate=180] \NR \end{scope}}
\def\NNNN{\NN \begin{scope}[rotate=180] \NN \end{scope}}
\def\NRS{\begin{scope}[shift={(B)}] \NR \end{scope}}
\def\NRD{\begin{scope}[rotate around={-90:(B)}] \NRS \end{scope}}
\def\DE{\draw plot [smooth,tension=1] coordinates {(0,0) (0.5,0.15) (1,0)}; \draw plot [smooth,tension=1.5] coordinates {(0,0) (0.5,-0.15) (1,0)};}
\def\DES{\begin{scope}[shift={(B)}] \DE \end{scope}}
\def\TE{\DE \draw (A) -- (B);}
\def\QE{\DE \draw plot [smooth,tension=1] coordinates {(A) (0.5,0.05) (B)}; \draw plot [smooth,tension=1.5] coordinates {(A) (0.5,-0.05) (B)};}
\def\T{\draw (0.2,0) -- (C) -- (B) -- (0.2,0);}
\def\TT{\draw (C) -- (B) -- (0.2,0); \draw plot [smooth,tension=1] coordinates {(0.2,0) (0.3,0.3) (C)}; \draw plot [smooth,tension=1] coordinates {(0.2,0) (0.5,0.2) (C)};}
\newcommand{\nn}[3]{\draw (#1)++(#2:3mm) node[fill=white,fill opacity=.85,inner sep=0mm,text=black,text opacity=1] {$\substack{\psi^#3}$};}
\renewcommand{\ggg}[2]{\fill (#2) circle(1.3mm) node {\color{white}$\substack #1$};}
\begin{definition}
  A stable graph $\Gamma$ in the sense of \cite[\S4.2]{pandharipande2015calculus} is a \emph{stable circular graph} if its vertices are connected in a closed chain.
  We denote the set of circular graphs by $G_{g, n}^\circ$.
\end{definition}
Any stable circular graph $\Gamma$ has the same number of vertices and edges.
For example, the following is a list of all stable circular graphs in $G_{4,0}^\circ$:
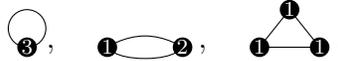
\begin{align*}
&\tikz{\NC \ggg{3}{A}}
, \quad \tikz{\DE \ggg{1}{A} \ggg{2}{B}},\quad  \tikz{\T \ggg{1}{B} \ggg{1}{0.2,0} \ggg{1}{C}}.
\end{align*}
\begin{theorem} 
\label{thm:LP-push-g-class}For $g\geq1$, the following topological recursion relation holds in  $R^{2g+1}(\overline{\mathcal{M}}_{g+1})$
\begin{align}\label{eqn:LP-push-g-class}
&\frac{2^{2g-1}}{(-1)^{g+1}(g+1)}\sum_{a+b=2g}(-1)^{a} \iota_*({\psi_{1}}^{a}{\psi_{2}}^{b})\nonumber
\\=&
\sum_{\Gamma\in G_{g,0}^\circ}\frac{1}{|\Aut(\Gamma)|}\sum_{\substack{\sum_{e\in E(\Gamma)}k(e)=g\\ k(e)\geq1, \forall e\in E(\Gamma)}}
(\xi_{\Gamma})_*\left(\prod_{e\in E(\Gamma)}\left(\frac{-1}{k(e)}\sum_{l(e)+m(e)=k(e)-1}\frac{\psi^{l(e)}_{h(e)}}{l(e))!}\frac{\psi_{h'(e)}^{m(e)}}{m(e)!}
 \right) \prod_{v\in V(\Gamma)}\mathsf{P}_{g(v)}^{g(v)}(0,0)\right)
\end{align}
where $\iota\colon \overline{\mathcal{M}}_{g,2} \to \overline{\mathcal M}_{g+1}$ and $\xi_{\Gamma}\colon \overline{\mathcal{M}}_{\Gamma}\rightarrow \overline{\mathcal{M}}_{g+1}$ are  the canonical gluing maps. $h(e)$ and $h'(e)$ denote the two half edges of $e$. 
\end{theorem}
Theorem~\ref{thm:LP-push-g-class} implies that
Equation~\eqref{eqn:LP-push-g} of Corollary~\ref{cor:LP-push-g} holds
for any cohomological field theory, as defined by Kontsevich and Manin
in \cite{kontsevich1994gromov}, in particular also for orbifold
Gromov-Witten theory and quantum singularity theory
(cf. \cite{chen2001orbifold}, \cite{fan2013witten}).
\subsection{Universal relations from Hodge class}
It was first mentioned in \cite{faber2000hodge} that a more sophisticated method of obtaining differential equations for  pure descendant Gromov-Witten invariants from
Faber--Pandharipande's formula (cf. \cite[Proposition 2]{faber2000hodge}) is to construct differential operators that correspond to multiplying by a Chern class of the Hodge bundle $\mathbb{E}$.
The Chern classes of Hodge bundle $\mathbb{E}$ certainly vanish
in degrees greater than $g$.
One obtains from \cite[Proposition~2]{faber2000hodge}) (Equation \eqref{eqn:FP-prop2} in this paper) relations
in degree greater than g (for each g).
It would be interesting to understand
these equations and their relation to topological recursion relations and the Virasoro constraints even
when $X$ is a point.

Let $z$ be a formal variable. We consider the space $\mathbb{H}_{+}$ which is the space of polynomials in one variable $z$ with coefficients in $H^*(X;\mathbb{C})$. Its has a canonical linear basis $\{z^n\phi_{\alpha}:\alpha=1,\dots,N, n\geq0\}$. Denote the coordinates on  $\mathbb{H}_{+}$ with respect to the canonical basis by $\{q_n^\alpha\}$. Then a general element in $\mathbb{H}_{+}$ can be written in the form $\mathbf{q}(z)=\sum_{n\geq0,\alpha}q_n^\alpha \phi_\alpha z^n$. For convenience, we identify $\mathbb{H}_{+}$ with the big phase space $\prod_{n=0}^{\infty}H^*(X;\mathbb{C})$ via identification between basis $\{z^{n}\phi_\alpha\}$ and $\{\tau_n(\phi_\alpha)\}$, and dilaton shift about coordinates $q_n^\alpha=\tilde{t}_n^\alpha$.
We will use the following convention
\begin{align*}
\widetilde{\mathcal{P}}(\mathbf{q})=\mathcal{P}(\mathbf{t})|_{q_n^\alpha=\Tilde{t}_n^\alpha}    
\end{align*}
for any partition function $\mathcal{P}(\mathbf{t})$.
Denote $\langle\langle\tau_{n_1}(\phi_{\alpha_1}),\dots,\tau_{n_k}(\phi_{\alpha_k})\rangle\rangle^{\sim}_g(\mathbf{q}(z))$ for the derivatives
of $\widetilde{F}_g(\mathbf{q})=F_g(\mathbf{t})_{q_n^\alpha=\Tilde{t}_n^\alpha}$ with respect to the variables $q_{n_1}^{\alpha_1},\dots,q_{n_k}^{\alpha_k}$.
Similarly, we use $\langle\langle\tau_{n_1}(\phi_{\alpha_1}),\dots,\tau_{n_k}(\phi_{\alpha_k});c_{g,k}\rangle\rangle^{\sim}_g(\mathbf{q}(z))$ to denote  the correlation function of Gromov--Witten theory twisted with a cohomology class $c_{g,k} \in H^*(\overline{\mathcal{M}}_{g,k};\mathbb{C})$.

\begin{theorem}[Theorem~\ref{thm:Hodge-class-relation}]
  \label{thm:intro-Hodge-class-relation}
  Let
  \begin{equation*}
    R(z)=\exp(\sum_{i=1}^{\infty}\frac{B_{2i}}{(2i)(2i-1)}u^{2i-1}z^{2i-1}).
  \end{equation*}
  Then, for $g\geq2$, the following holds for descendant Gromov--Witten invariants
  \begin{align*}
  [u^i]\sum_{\Gamma\in G^{\mathrm{Feyn}}_{g}}\frac{1}{|\Aut(\Gamma)|}\Cont_{\Gamma}
  =\begin{cases}
0, &i>g
\\
\frac{(-1)^g}{2^g}\cdot\langle\langle \,\,;\mathsf{P}_{g}^g(\emptyset)\rangle\rangle_g^{\sim}(\mathbf{q}(z))
&i=g
\end{cases}
 \end{align*}
where $[u^i]$ denotes taking the coefficient of $u^i$ and the contribution $\Cont_{\Gamma}$ of a Feynman graph $\Gamma$ is defined as a contraction of tensors:
\begin{itemize}
\item{Each vertex $v$ is assigned a tensor bracket
$$\langle\langle\quad\rangle\rangle_{g(v)}^{\sim}(R^{-1}\mathbf{q})$$
with insertion coming from the half edges.}
\item{Each edge $e$ is assigned a bivector with descendants $\frac{\sum_{\alpha,\beta}\left(\eta^{\alpha\beta}\phi_{\alpha}\otimes\phi_{\beta}-\eta^{\alpha\beta} R^{-1}(\psi_{h(e)})(\phi_{\alpha})\otimes R^{-1}(\psi_{h'(e)})(\phi_{\beta})\right)
}{\psi_{h(e)}+\psi_{h'(e)}}$}
\end{itemize}
\end{theorem}
\begin{proposition}[see Proposition~\ref{prop:3-spin}]
  \label{prop:intro-3-spin}
  Pixton's 3-spin relations \cite{pixton2012conjectural} on the moduli space of stable curves imply the $i>g$ part of Theorem~\ref{thm:Hodge-class-relation}.
\end{proposition}

\subsection{Plan of the paper}

This paper is organized as follows.
In Section~\ref{sec:chern-character-Hodge}, we review Faber and
Pandharipande's formula for Hodge classes, Pixton's double
ramification cycle formula for the top Hodge class, and prove
Theorem~\ref{thm:deg-g-chern-chara} and Corollary~\ref{cor:LP-push-g}.
In Section~\ref{sec:push-forward relation}, we give prove the
new push-forward relation of Theorem~\ref{thm:LP-push-g-class}.
In Section~\ref{sec:Chern class of Hodge bundle}, we prove
Theorem~\ref{thm:intro-Hodge-class-relation} and
Proposition~\ref{prop:intro-3-spin}.

\section{Chern characters of the Hodge bundle}\label{sec:chern-character-Hodge}
\subsection{Review of Faber and Pandharipande's formula}

Let $\overline{\mathcal{M}}_{g,n}(X,\beta)$ be the moduli space of stable maps $\{f: (C; p_1,\dots,p_n)\rightarrow X\}$, where $(C; p_1,\dots,p_n)$
is a genus-$g$ nodal curve with $n$ marked points and the image $f_*[C]=\beta\in H_2(X;\mathbb{Z})$ is the effective class $\beta$.
The Hodge bundle $\mathbb{E}$ over $\overline{\mathcal{M}}_{g,n}(X,\beta)$ is the rank $g$ vector bundle with fiber $H^0(C,\omega_{C})$ over a moduli point
$f: (C,p_1,\dots,p_n)\rightarrow X$, where $\omega_{C}$ is the dualizing sheaf of the domain curve $C$.
In fact, the Hodge bundle over $\overline{\mathcal{M}}_{g,n}(X,\beta)$ is the pullback of the Hodge bundle from $\overline{\mathcal{M}}_{g,n}$ under the stablization map $\St\colon \overline{\mathcal{M}}_{g,n}(X,\beta)\rightarrow \overline{\mathcal{M}}_{g,n}$.
Let $\ch(\mathbb{E})$ be
the Chern character of $\mathbb{E}$ and denote its degree $k$ component  by $\ch_{k}(\mathbb{E})$.
It is well-known that the positive even Chern characters $\ch_{2k}(\mathbb{E})$ vanish (cf. \cite{mumford1983towards}). Let $\lambda_i:=c_i(\mathbb{E})$ be the $i$-th Chern class of Hodge bundle. 

In \cite{faber2000hodge}, Faber and Pandharipande considered the following Hodge integrals over the moduli space of stable maps:
\begin{align*}
\big\langle\prod_{i=1}^{n}\tau_{k_i}(\phi_{\alpha_i});\prod_{j=1}^{m}\ch_{l_j}(\mathbb{E})\big\rangle_{g,\beta}:=\int_{[\overline{\mathcal{M}}_{g,n}(X,\beta)]^\vir}\prod_{i=1}^{n}\psi_i^{k_i}\prod_{i=1}^{n}\ev_i^*\phi_{\alpha_i}\prod_{j=1}^{m}\ch_{l_j}(\mathbb{E})
\end{align*}
where $[\overline{\mathcal{M}}_{g,n}(X,\beta)]^\vir$ is the
virtual fundamental class of $\overline{\mathcal{M}}_{g,n}(X,\beta)$,
the descendant $\psi$-class $\psi_i$ is the first Chern class of the
tautological line bundle over $\overline{\mathcal{M}}_{g,n}(X,\beta)$
defined by the cotangent lines at the $i$-th marking of the domain
curves, and $\ev_i\colon \overline{\mathcal{M}}_{g,n}(X,\beta) \to X$
is the evaluation map defined by the image of the $i$-th marking.

The generating
function of genus-$g$ Hodge integrals is defined to be
\begin{align}\label{eqn:Fg-Hodge-t-s}
F_g^{\mathbb{E}}(\mathbf{t},\mathbf{s})
:=\sum_{\beta\in H_2(X;\mathbb{Z})}Q^\beta
\Bigg\langle\exp\left(\sum_{n,\alpha}t_n^\alpha\tau_n(\phi_\alpha)\right);\exp\left(\sum_{m\geq1}s_{2m-1}\ch_{2m-1}(\mathbb{E})\right)\Bigg\rangle_g
\end{align}
where $Q$ is the Novikov variable. 
In particular, after the restriction $\mathbf{s}=0$, the generating
function $F_g(\mathbf{t}):=F_{g}^{\mathbb{E}}|_{\mathbf{s}=0}$ becomes
the usual generating function of genus-$g$ pure descendant
Gromov--Witten invariants. For convenience, we identify $\tau_n(\phi_\alpha)$ with $\frac{\partial}{\partial t_n^\alpha}$
and $\ch_n(\mathbb{E})$ with $\frac{\partial}{\partial s_n}$. Moreover we use
the convention that $\tau_0(\phi_\alpha)=\phi_\alpha$ and $\tau_n(\phi_\alpha)=0$ if $n<0$. Double brackets denote differentiation of $F_{g}^{\mathbb{E}}(\mathbf{t},\mathbf{s})$,
\begin{align*}
\langle\langle\tau_{k_1}(\phi_{\alpha_1})\ldots\tau_{k_n}(\phi_{\alpha_n});\prod_{i=1}^{m}\ch_{2l_i-1}(\mathbb{E})\rangle\rangle_{g}^{\mathbb{E}} 
=\frac{\partial^{n+m}}{\partial t_{k_1}^{\alpha_1}\dots\partial t_{k_n}^{\alpha_n}\partial s_{2l_1-1}\dots\partial s_{2l_m-1}}F_{g}^{\mathbb{E}}(\mathbf{t},\mathbf{s})
\end{align*}
and
\begin{align*}
\langle\langle\tau_{k_1}(\phi_{\alpha_1})\ldots\tau_{k_n}(\phi_{\alpha_n});\prod_{i=1}^{m}\ch_{2l_i-1}(\mathbb{E})\rangle\rangle_{g}
=\langle\langle\tau_{k_1}(\phi_{\alpha_1})\ldots\tau_{k_n}(\phi_{\alpha_n});\prod_{i=1}^{m}\ch_{2l_i-1}(\mathbb{E})\rangle\rangle_{g}^{\mathbb{E}}|_{\mathbf{s}=0}.    
\end{align*}

In \cite{faber2000hodge}, the following universal equation for descendant  Gromov--Witten invariants of any smooth projective variety was proved for any $l>g$:
\begin{align}\label{eqn:FP-l>g}
-\Tilde{t}_n^\alpha\langle\langle\tau_{n+2l-1}(\phi_\alpha)\rangle\rangle_g+\frac{1}{2}\sum_{i=0}^{2l-2}(-1)^{i}
\left\{\langle\langle\tau_i(\phi_\alpha)\tau_{2l-2-i}(\phi^\alpha)\rangle\rangle_{g-1}
+\sum_{h=0}^{g}\langle\langle\tau_{i}(\phi_\alpha)\rangle\rangle_h
\langle\langle\tau_{2l-2-i}(\phi^\alpha)\rangle\rangle_{g-h}\right\}=0
\end{align}
where $\tilde{t}_n^\alpha=t_n^\alpha-\delta_{n}^{1}\delta_\alpha^{1}$. In fact, the Grothendieck--Riemann--Roch formula on the moduli space of
stable maps (cf. \cite{faber2000hodge}) implies
\begin{align}\label{eqn:grr-FP-l>g}
&\frac{(2l)!}{B_{2l}}
\langle\langle \ch_{2l-1}(\mathbb{E})\rangle\rangle_{g}^{\mathbb{E}}
\nonumber\\=&-\Tilde{t}_n^\alpha\langle\langle\tau_{n+2l-1}(\phi_\alpha)\rangle\rangle_g^{\mathbb{E}}+\frac{1}{2}\sum_{i=0}^{2l-2}(-1)^{i}
\left\{\langle\langle\tau_i(\phi_\alpha)\tau_{2l-2-i}(\phi^\alpha)\rangle\rangle_{g-1}^{\mathbb{E}}
+\sum_{h=0}^{g}\langle\langle\tau_{i}(\phi_\alpha)\rangle\rangle_h^{\mathbb{E}}
\langle\langle\tau_{2l-2-i}(\phi^\alpha)\rangle\rangle_{g-h}^{\mathbb{E}}\right\}
\end{align}
where $B_{2n}$ are the classical Bernoulli numbers defined by
\[\frac{x}{e^x-1}=1+\frac{x}{2}+\sum_{n=1}^{\infty}B_{2n}\frac{x^{2n}}{(2n)!}.\]
Then \eqref{eqn:FP-l>g} follows from \eqref{eqn:grr-FP-l>g}, the vanishing of Chern character of Hodge bundle $\ch_{2l-1}(\mathbb{E})=0$ for $l>g$, and setting $\mathbf{s}=0$.
This leads to the question how can these universal equations for pure descendant Gromov--Witten invariants be extended to the case $l\leq g$?
In this paper, we give an answer to this question in the case $l=g$.

\subsection{Pixton's formula for double ramification cycles and the $\lambda_g$ class} \label{sec:pixfor}
\subsubsection{Pixton's formula for double ramification cycles}
\label{pixconj}

Denote $G_{g,n}$ to be the set of stable graphs of genus-$g$, $n$ marked points \cite[\S4.2]{pandharipande2015calculus}.
Let $A=(a_1,\ldots,a_n)$ be a vector of double ramification data,
i.e.\ a vector of integers such that $\sum_{i=1}^na_i=0$.
Let $\Gamma \in G_{g,n}$ be a stable graph of genus $g$ with $n$ legs
and $r$ be a positive integer.
A weighting mod $r$ of $\Gamma$ is a function on the set of
half-edges,
$$ w\colon H(\Gamma) \rightarrow \{0,\dots,r-1\},$$
which satisfies the following three properties:
\begin{enumerate}
\item[(i)] $\forall h_i\in L(\Gamma)$, corresponding to
 the marking $i\in \{1,\ldots, n\}$,
$$w(h_i) \equiv a_i \mod r \ ,$$
\item[(ii)] $\forall e \in E(\Gamma)$, corresponding to two half-edges
$h,h' \in H(\Gamma)$,
$$w(h)+w(h') \equiv 0 \mod r\, ,$$
\item[(iii)] $\forall v\in V(\Gamma)$,
$$\sum_{v(h)= v} w(h) \equiv 0 \mod r\, ,$$ 
where the sum is taken over all  half-edges incident to $v$.
\end{enumerate}
Let $\mathsf{W}_{\Gamma,r}$ be the set of weightings mod $r$
of $\Gamma$. The set $\mathsf{W}_{\Gamma,r}$ is finite, with cardinality $r^{h^1(\Gamma)}$.
For each stable graph $\Gamma\in G_{g,n}$, we associate  a moduli space $\overline{\mathcal{M}}_\Gamma:=\prod_{v\in V(\Gamma)}\overline{\mathcal{M}}_{g(v),n(v)}$ and define a natural map $\xi_{\Gamma}: \overline{\mathcal{M}}_\Gamma \to 
\overline{\mathcal{M}}_{g,n}$ to be the  canonical gluing morphism.

We denote by
$\mathsf{P}_g^{d,r}(A)\in R^d(\overline{\mathcal{M}}_{g,n})$ the degree $d$ component of the tautological class 
\begin{align*}
\hspace{-5pt}\sum_{\Gamma\in {G}_{g,n}} 
\sum_{w\in \mathsf{W}_{\Gamma,r}}
\frac{1}{|\Aut(\Gamma)| } 
\frac{1}{r^{h^1(\Gamma)}}
\xi_{\Gamma*}\Bigg[
\prod_{i=1}^n \exp(a_i^2 \psi_{h_i}) \cdot 
\prod_{e=(h,h')\in E(\Gamma)}
\frac{1-\exp(-w(h)w(h')(\psi_h+\psi_{h'}))}{\psi_h + \psi_{h'}} \Bigg]
\end{align*} 
in $R^*(\overline{\mathcal{M}}_{g,n})$.

Pixton proved for fixed $g$, $A$, and $d$, the \label{pply}
class
$\mathsf{P}_g^{d,r}(A) \in R^d(\overline{\mathcal{M}}_{g,n})$
is polynomial in $r$ for $r$ sufficiently large. 
We denote by $\mathsf{P}_g^d(A)$ the value at $r=0$ 
of the polynomial associated to $\mathsf{P}_g^{d,r}(A)$. In other words, $\mathsf{P}_g^d(A)$ is the  constant term of the associated polynomial in $r$. 

The following formula for double ramification cycles was first conjectured
by Pixton.
\begin{theorem}[\cite{janda2017double}] \label{thm:DR-Pixton}
For $g\geq 0$ and double ramification data $A$, we have
$$\mathsf{DR}_g(A) = 2^{-g}\, \mathsf{P}_g^g(A)\, \in R^g(\overline{\mathcal{M}}_{g,n}).$$ 
\end{theorem}

\subsubsection{Formula for $\lambda_g$}

The top Chern class $\lambda_g$ of the Hodge bundle $\EE$ is a very
special case of a double ramification cycle (see
\cite[\S0.5.3]{janda2017double}):
\begin{equation}
\label{eqn:DR-lambda-g}\mathsf{DR}_g(0,\ldots,0) = (-1)^g\lambda_g \ \in R^g(\overline{\mathcal{M}}_{g,n}).\    
\end{equation}

Combining Theorem~\ref{thm:DR-Pixton} and \eqref{eqn:DR-lambda-g}, it
leads to a nice formula for the top Chern class of Hodge bundle
$\lambda_g$, which is supported on the boundary divisor of curves with
a non-separating node.
That is, in $R^{g}(\overline{\mathcal{M}}_{g,n})$, we have
\begin{equation}\label{eqn:lambda-g-formula}
  \lambda_g=\frac{(-1)^g}{2^g} \mathsf{P}_{g}^{g}(0,\dots,0) .
\end{equation} 
To illustrate this, we list some explicit formulae of $\lambda_g$ on
$\overline{\mathcal{M}}_g$ $(g\geq2)$ in low genus
(cf. \cite[\S3.1]{janda2017double}, which also lists a longer formula
in genus $4$):

\paragraph{Genus~1.}
\begin{align}
  \label{eqn:P_1^1(00)}
  \lambda_1 = -\frac 12 \mathsf{P}_{1}^{1}(0,0) = \frac{1}{24} \xi_{\Gamma*}(1)  
\end{align}
where $\Gamma \in G_{1,2}$ is the dual graph of the boundary divisor of singular stable curves
with a non-separating node.

\paragraph{Genus~2.}
\begin{equation*}
\lambda_2 = 
  \frac 1{240} \tikz{\NC \nn{A}{130}{{}} \ggg{1}{A}}
  + \frac 1{1152} \tikz{\NN \ggg{0}{A}}.
\end{equation*}

\paragraph{Genus~3.}
\begin{align*}
\lambda_3 &= 
  \frac 1{2016} \tikz{\NC \nn{A}{130}{2} \ggg{2}{A}}
  + \frac 1{2016} \tikz{\NC \nn{A}{130}{{}} \nn{A}{50}{{}} \ggg{2}{A}}
  - \frac 1{672} \tikz{\DE \nn{A}{30}{{}} \ggg{1}{A} \ggg{1}{B}}
  + \frac 1{5760} \tikz{\NN \nn{A}{160}{{}} \ggg{1}{A}} \\
  &
  - \frac{13}{30240} \tikz{\TE \ggg{0}{A} \ggg{1}{B}}
  - \frac 1{5760} \tikz{\NL \DE \ggg{0}{A} \ggg{1}{B}}
  + \frac 1{82944} \tikz{\NNN \ggg{0}{A}}.
\end{align*}
All these expressions are obtained by substituting $A=\emptyset$ in
Pixton's formula for the double ramification cycle (except in the
genus one case, where we should set $A = (0, 0)$).

\subsection{Universal equations from $\ch_{2g-1}(\mathbb{E})$}

A well-known formula for $\ch_{2g-1}(\mathbb{E})$ (cf. \cite[(46)]{faber2000logarithmic}) is
\begin{align}\label{eqn:ch-2g-1-hodge}
\ch_{2g-1}(\mathbb{E})=\frac{(-1)^{g-1}}{(2g-1)!}\lambda_g\lambda_{g-1}.
\end{align}
In fact, this is a consequence from the following general relationship between the Chern characters and Chern classes of a vector bundle:
\begin{align*}
\sum_{k\geq1}(-1)^{k-1}(k-1)!\ch_k(\mathbb{E})t^k=\log\left(\sum_{k\geq0}c_{k}(\mathbb{E})t^k\right).
\end{align*}
Taking derivative $\frac{d}{dt}$, we get
\begin{align*}
\sum_{k\geq1}(-1)^{k-1}k!\ch_k(\mathbb{E})t^{k-1}=\left(\sum_{k\geq0}\lambda_{k}t^k\right)^{-1}
\cdot \left(\sum_{k\geq1}k\cdot\lambda_{k}t^{k-1}\right).
\end{align*}
Since $c(\mathbb{E})^{-1}=c(\mathbb{E}^*)$, the above becomes
\begin{align}\label{eqn:chern-chara-chern-class-Hodge}
\sum_{k\geq1}(-1)^{k-1}k!\ch_k(\mathbb{E})t^{k-1}=\left(\sum_{k=0}^{g}\lambda_{k}(-t)^k\right)
\cdot \left(\sum_{k=1}^{g}k\cdot\lambda_{k}t^{k-1}\right).
\end{align}
Taking the coefficients of $t^{2g-2}$ on both sides of \eqref{eqn:chern-chara-chern-class-Hodge} gives \eqref{eqn:ch-2g-1-hodge}.

Recall the following  important properties of the Hodge bundle:
\begin{itemize}
    \item Pulling back via the gluing map $\iota_{g_1,g_2}\colon \overline{\mathcal{M}}_{g_1,n_1+1}\times \overline{\mathcal{M}}_{g_2,n_2+1}\rightarrow \overline{\mathcal{M}}_{g,n}$, we have
    \begin{align}
    \label{eqn:prop-Hodge-1}\iota_{g_1,g_2}^*\mathbb{E}_g\cong p_1^*\mathbb{E}_{g_1}\oplus p_2^*\mathbb{E}_{g_2}\end{align}
    where $p_i (i=1,2)$ denotes via the projection maps from $\overline{\mathcal{M}}_{g_1,n_1+1}\times \overline{\mathcal{M}}_{g_2,n_2+1}$ onto its factors.
    \item Pulling back via the gluing map: $\iota_{g-1}\colon \overline{\mathcal{M}}_{g-1,n+1}\rightarrow \overline{\mathcal{M}}_{g,n}$, we have
    \begin{align}
    \label{eqn:prop-Hodge-2}\iota_{g-1}^*\mathbb{E}_g\cong \mathbb{E}_{g-1}\oplus\mathcal{O}\end{align}
    where $\mathcal{O}$ is the structure sheaf.
\end{itemize}

\begin{lemma}\label{lem:con-gamma-lambda-g}
  For any $g\geq2$ and a stable circular graph $\Gamma\in {G}_{g,0}^\circ$, the contribution of $\Gamma$ to $\mathsf{P}_{g}^{g}(\emptyset)$ is:
  \begin{align}\label{eqn:con-gamma-lambda-g}
    &\Cont_{\Gamma}\mathsf{P}_{g}^g(\emptyset)= 
      \frac{B_{2g}}{|\Aut(\Gamma)|}\sum_{\substack{\sum_{e\in E(\Gamma)=g}k(e)=g\\ k(e)\geq1, \forall e\in E(\Gamma)}
    }(\xi_{\Gamma})_*\prod_{e\in E(\Gamma)}\left(\frac{-1}{k(e)!}(\psi_{h(e)}+\psi_{h'(e)})^{k(e)-1}  \right).
  \end{align}
\end{lemma}
\begin{proof}
  For any graph $\Gamma\in G_{g,0}^\circ$, the associated weighting modulo $r$ is either identically equal to $0$, or alternatingly $a$ and $r - a$ for some $1\leq a\leq r-1$.
  From \cite[Appendix]{janda2017double}, we know that
  \begin{equation*}
  \left[\frac{1}{|\Aut(\Gamma)|}
\frac{1}{r}\sum_{a=1}^{r-1}(\xi_{\Gamma})_*\prod_{e\in E(\Gamma)}\frac{1-\exp(-a(r-a)(\psi_{h(e)}+\psi_{h'(e)}))}{\psi_{h(e)}+\psi_{h'(e)}}\right]_{\deg=g}
  \end{equation*}
  is polynomial in $r$ for $r$ sufficiently large. 
  The contribution $\Cont_{\Gamma}\mathsf{P}_{g}^g(\emptyset)$ of the graph
  $\Gamma$ to $\mathsf{P}_{g}^{g}(\emptyset)$ equals the constant term of this polynomial.
  By a direct computation,
  \begin{align*}
    &\left[\frac{1}{|\Aut(\Gamma)|}
      \frac{1}{r}\sum_{a=1}^{r-1}(\xi_{\Gamma})_*\prod_{e\in E(\Gamma)}\frac{1-\exp(-a(r-a)(\psi_{h(e)}+\psi_{h'(e)}))}{\psi_{h(e)}+\psi_{h'(e)}}   \right]_{\deg=g}
    \\=&
         \frac{1}{|\Aut(\Gamma)|}\frac{1}{r}\sum_{a=1}^{r-1}\sum_{\sum_{e\in E(\Gamma)}k(e)=g
         }(\xi_{\Gamma})_*\prod_{e\in E(\Gamma)}\frac{-1}{k(e)!}a^{k(e)}(a-r)^{k(e)}(\psi_{h(e)}+\psi_{h'(e)})^{k(e)-1}  
    \\=&
         \left(\frac{1}{r}\sum_{a=1}^{r-1}a^g(a-r)^{g}\right)\cdot\frac{1}{|\Aut(\Gamma)|}\sum_{\sum_{e\in E(\Gamma)}k(e)=g
         }(\xi_{\Gamma})_*\prod_{e\in E(\Gamma)}\frac{-1}{k(e)!}(\psi_{h(e)}+\psi_{h'(e)})^{k(e)-1}. 
  \end{align*}
  Therefore,
  \begin{equation*}
    \Cont_{\Gamma}\mathsf{P}_{g}^g(\emptyset) = \frac{B_{2g}}{|\Aut(\Gamma)|}\sum_{e\in E(\Gamma)=g
    }(\xi_{\Gamma})_*\prod_{e\in E(\Gamma)}\left(\frac{-1}{k(e)!}(\psi_{h(e)}+\psi_{h'(e)})^{k(e)-1}  \right),
  \end{equation*}
  where we have used Faulhaber's formula 
  \begin{align*}
    \sum_{a=1}^{r-1}a^{m}=\frac{1}{m+1}\sum_{k=0}^{m}\binom{m+1}{k}B_{k}r^{m-k+1}    
  \end{align*}
  for any $m\geq1$.
\end{proof}

\begin{proposition}
\label{prop:lambda-g-lambda-g-1}
For any $g\geq2$, the following relation holds on $\overline{\mathcal{M}}_{g}$:
\begin{small}
\begin{align}\label{eqn:lambda-g-g-1}
&\lambda_g\lambda_{g-1}
\nonumber\\=&\frac{-B_{2g}}{2^{2g-1}}
\sum_{\Gamma\in G_{g,0}^\circ}\frac{1}{|\Aut(\Gamma)|}\sum_{\substack{\sum_{e\in E(\Gamma)}k(e)=g\\ k(e)\geq1, \forall e\in E(\Gamma)}}
(\xi_{\Gamma})_*\left(\prod_{e\in E(\Gamma)}\left(\frac{-1}{k(e)}\sum_{l(e)+m(e)=k(e)-1}\frac{\psi^{l(e)}_{h(e)}}{l(e)!}\frac{\psi_{h'(e)}^{m(e)}}{m(e)!}
 \right) \prod_{v\in V(\Gamma)}\mathsf{P}_{g(v)}^{g(v)}(0,0)\right)
\end{align}   
\end{small}
\end{proposition}
\begin{proof}
  We compute:
  \begin{align*}
&\lambda_{g-1}\cdot\mathsf{P}_{g}^{g}(\emptyset)
=\lambda_{g-1}\cdot \sum_{\Gamma\in {G}_{g,0}}  \Cont_{\Gamma}\mathsf{P}_{g}^g(\emptyset)
=\sum_{\Gamma\in {G}_{g,0}^\circ} \lambda_{g-1}\cdot \Cont_{\Gamma}\mathsf{P}_{g}^g(\emptyset)
\\=&\sum_{\Gamma\in {G}_{g,0}^\circ} \frac{B_{2g}}{|\Aut(\Gamma)|}\sum_{\substack{\sum_{e\in E(\Gamma)}k(e)=g\\ k(e)\geq1 \forall e\in E(\Gamma)}
}(\xi_{\Gamma})_*\left(\prod_{e\in E(\Gamma)}\frac{-1}{k(e)!}(\psi_{h(e)}+\psi_{h'(e)})^{k(e)-1}  \cdot \prod_{v\in V(\Gamma)}\lambda_{g(v)}\right)
  \end{align*}
  In the second equality, we used both that, first, by definition, no
  stable graph with at least one separating edge contributes to
  $\mathsf{P}_g^g(\emptyset)$.
  Second, no graph with more than one loop contributes, since applying
  the properties of the Hodge bundle \eqref{eqn:prop-Hodge-1} and
  \eqref{eqn:prop-Hodge-2} to $\lambda_{g - 1}$, the contribution of
  such a graph would involve a Hodge class $\lambda_i$ for $i > g(v)$
  on some vertex moduli space $\overline{\mathcal{M}}_{g(v), n(v)}$.
  Finally, in the third equality, we again used the properties of the
  Hodge bundle, combined with Lemma~\ref{lem:con-gamma-lambda-g}.
  We next apply \eqref{eqn:lambda-g-formula} twice, and obtain
\begin{small}
\begin{align*}
&\lambda_g\lambda_{g-1} =\frac{(-1)^g}{2^g}\lambda_{g-1} \mathsf{P}_{g}^{g}(\emptyset)
\\=&\frac{(-1)^g}{2^g}\sum_{\Gamma\in {G}_{g,0}^\circ} \frac{B_{2g}}{|\Aut(\Gamma)|}\sum_{\substack{\sum_{e\in E(\Gamma)}k(e)=g\\ k(e)\geq1 \forall e\in E(\Gamma)}
}(\xi_{\Gamma})_*\left(\prod_{e\in E(\Gamma)}\frac{-1}{k(e)!}(\psi_{h(e)}+\psi_{h'(e)})^{k(e)-1}  \cdot \prod_{v\in V(\Gamma)}\lambda_{g(v)}\right)
\\=&\frac{(-1)^{2g-1}}{2^{2g-1}}\sum_{\Gamma\in {G}_{g,0}^\circ} \frac{B_{2g}}{|\Aut(\Gamma)|}\sum_{\substack{\sum_{e\in E(\Gamma)}k(e)=g\\ k(e)\geq1 \forall e\in E(\Gamma)}
}(\xi_{\Gamma})_*\left(\prod_{e\in E(\Gamma)}\frac{-1}{k(e)!}(\psi_{h(e)}+\psi_{h'(e)})^{k(e)-1}  \cdot \prod_{v\in V(\Gamma)}\mathsf{P}^{g(v)}_{g(v)}(0,0)\right).
\end{align*}    
\end{small}
After expanding each factor $(\psi_{h(e)}+\psi_{h'(e)})^{k(e)-1}$ via the binomial theorem, this yields \eqref{eqn:lambda-g-g-1}. 
\end{proof}

% There are many such graphs, denote $\Gamma_m^{(g_1,\dots,g_m)}$, then by Lemma~\ref{lem:con-gamma-lambda-g}
% \begin{align*}
% &\lambda_{g-1}\cdot Cont_{\Gamma_{m}^{(g_1,\dots,g_m)}}\mathcal{D}_{g,0}^{g}
% \\=&
% B_{2g}
% \frac{1}{|Aut(\Gamma_m)|}\sum_{k_1+\dots+k_m=g
% }(\xi_{\Gamma_m})_*\prod_{i=1}^{m}\left(\frac{-1}{2^{k_i}k_i!}\sum_{l_i=0}^{k_i-1}\binom{k_i-1}{l_i}{\psi_{h(e_i)}}^{l_i}{\psi_{h'(e_i)}}^{k_i-1-l_i}\cdot \lambda_{g_i}\right) 
% \end{align*}

% Using the formula \eqref{eqn:lambdag-formula} for $\lambda_g$ class, we see $\lambda_g\lambda_{g-1}$ can be expressed in terms of product of $\psi$ classes on the boundary.
% \begin{align*}
% \lambda_g\lambda_{g-1}=\lambda_{g-1}\sum_{\Gamma\in G_{g,n}: h^1(\Gamma)=1}\dots
% =\sum_{m=1}^{g-1}\sum_{\substack{g_1+\dots+g_m=g-1\\ g_1,\dots,g_m\geq1}}(\xi_{\Gamma_m})_*\left[\prod_{i=1}^{m}DR_{g_i}(a,-a)\right]_{a^0}
% \end{align*} 
% if we work on $\overline{M}_{g}$ ($g\geq2$), 
% there are only $g-1$ types of graphs in the above sum: 
% \begin{itemize}
%     \item  the graph with a vertex of genus $g-1$ and one loop
%     \item the graph with 2 vertices of genus $g_1$ an $g_2$ and one loop
%   \item  \dots
%    \item the graph with  $g-1$ vertices of genus $1,\dots,1$ and one loop. 
% \end{itemize}
% \begin{align*}
% &\sum_{m=1}^{g-1}\sum_{\substack{g_1+\dots+g_m=g-1\\ g_1,\dots,g_m\geq1}}
% \prod_{i=1}^{m}\langle\langle\bar{\psi}^{..}\phi_{\alpha_i
% },\bar{\psi}^{\dots}\phi^{\alpha_{i+1}};\lambda_{g_i}\rangle\rangle_{g_i} 
% \end{align*}
To illustrate this, we list formulae for $\lambda_g\lambda_{g-1}$ on $\overline{\mathcal{M}}_g$ in low genus:

\paragraph{Genus~2.}
\begin{equation*}
\lambda_2\lambda_1 = 
  -\frac{1}{480} \tikz{\NC \nn{A}{130}{{}} \ggg{1}{A}}_{\mathsf{P}_1^1(0,0)}.
\end{equation*}
\paragraph{Genus~3.}
\begin{align*}
\lambda_3 \lambda_2&= 
  \frac{1}{8064}  \tikz{\NC \nn{A}{130}{2} \ggg{2}{A}}_{\mathsf{P}_2^2(0,0)}
  + \frac{1}{8064}  \tikz{\NC \nn{A}{130}{{}} \nn{A}{50}{{}} \ggg{2}{A}}_{\mathsf{P}_2^2(0,0)}
  - \frac{1}{2688}  \empty_{\mathsf{P}_1^1(0,0)}\tikz{\DE \nn{A}{30}{{}} \ggg{1}{A} \ggg{1}{B}}_{\mathsf{P}_1^1(0,0)}.
\end{align*}
\paragraph{Genus~4.}
\begin{align*}
\lambda_4 \lambda_3&= 
  -\frac{1}{15360}  \tikz{\NC \nn{A}{130}{3} \ggg{3}{A}}_{\mathsf{P}_3^3(0,0)}
  - \frac{1}{5120} \tikz{\NC \nn{A}{130}{2} \nn{A}{50}{{}} \ggg{3}{A}}_{\mathsf{P}_3^3(0,0)}
  +\frac{1}{23040}  _{\mathsf{P}_1^1(0,0)}\tikz{\DE \nn{A}{30}{2} \ggg{1}{A} \ggg{2}{B}}_{\mathsf{P}_2^2(0,0)}
 +\frac{1}{30720} _{\mathsf{P}_1^1(0,0)}\tikz{\DE \nn{A}{30}{{}} \nn{A}{-30}{{}} \ggg{1}{A} \ggg{2}{B}}_{\mathsf{P}_2^2(0,0)}\\&
  +\frac{1}{11520} _{\mathsf{P}_1^1(0,0)}\tikz{\DE \nn{A}{30}{{}} \nn{B}{150}{{}} \ggg{1}{A} \ggg{2}{B}}_{\mathsf{P}_2^2(0,0)}+\frac{1}{15360}_{\mathsf{P}_1^1(0,0)}\tikz{\DE \nn{A}{30}{{}} \nn{B}{-150}{{}} \ggg{1}{A} \ggg{2}{B}}_{\mathsf{P}_2^2(0,0)}
  + \frac{1}{23040} _{\mathsf{P}_1^1(0,0)}\tikz{\DE \nn{B}{150}{2} \ggg{1}{A} \ggg{2}{B}}_{\mathsf{P}_2^2(0,0)}
  \\&+\frac{1}{30720} _{\mathsf{P}_1^1(0,0)}\tikz{\DE \nn{B}{150}{{}} \nn{B}{-150}{{}} \ggg{1}{A} \ggg{2}{B}}_{\mathsf{P}_2^2(0,0)}
  -\frac{1}{7680}_{\mathsf{P}_1^1(0,0)}\tikz{\T \nn{B}{180}{{}} \ggg{1}{B} \ggg{1}{0.2,0} \ggg{1}{C} \draw (C) node [anchor=south] {$\substack{\mathsf P_1^1(0,0)}$}}_{\mathsf{P}_1^1(0,0)}.
\end{align*}

\subsection{Proof of Theorem \ref{thm:deg-g-chern-chara}}
We first recall some standard results about the correspondence between descendant and ancestor invariants (cf. \cite[Appendix 2]{coates2007quantum}). 
Consider the stablization morphism
\[\St\colon \overline{\mathcal{M}}_{g,n}(X,\beta)\rightarrow \overline{\mathcal{M}}_{g,n}\]
which forgets the map $f$ in $(C; x_1,\dots,x_n; f ) \in  \overline{\mathcal{M}}_{g,n}(X,\beta)$ and stabilizes the curve
$(C; x_1,\dots,x_n)$
by contracting unstable components to points. 
Let
\[\pi_m\colon \overline{\mathcal{M}}_{g,n+m}(X,\beta)\rightarrow \overline{\mathcal{M}}_{g,n}(X,\beta)\]
be the morphism, which forgets the last $m$ markings, and stabilizes the curve.
For $1\leq i\leq n$, denote the  cotangent line bundle  along the $i$-th marked point over $\overline{\mathcal{M}}_{g,n}$ and $\overline{\mathcal{M}}_{g,n+m}(X,\beta)$ by $L_i$ and $\mathcal{L}_i$, respectively.
The bundles $\mathcal{L}_i$ and $\pi_m^*\St^*L_i$  over
$\overline{\mathcal{M}}_{g,n+m}(X,\beta)$ are identified outside the locus $D$ consisting of maps such that the $i$-th marked
point is situated on a component of the curve which gets collapsed by $\St \circ \pi_m$.

This locus $D$ is the image of the gluing map
\begin{align*}
   \iota\colon: \sqcup_{\substack{m_1+m_2=m\\ \beta_1+\beta_2=d}} \overline{\mathcal{M}}_{0,\{i\}+\bullet+m_1}(X,\beta_1)\times_{X} \overline{\mathcal{M}}_{g,[n]\setminus\{i\}+\circ+m_2}(X,\beta_2)\rightarrow \overline{\mathcal{M}}_{g,n+m}(X,\beta)
\end{align*}
where the two markings $\bullet$ and $\circ$ are glued together under map $\iota$.
We denote the domain of this map by $Y_{n,m,\beta}^{(i)}$. The virtual normal bundle to $D$ at a generic
point is $\Hom(\pi_m^*\St^*c_1(L_i), \mathcal{L}_i)$, and so $D$ is ``virtually Poincar\'e-dual'' to $c_1(\mathcal{L}_i)-\pi_m^*\St^*c_1(L_i)$ in the sense that
\begin{align}
\label{eqn:D-Y-vir}  \left(c_1(\mathcal{L}_i)-\pi_m^*\St^*c_1(L_i)\right) \cap [\overline{\mathcal{M}}_{n+m}{(X,\beta)}]^\vir =\iota_*[Y_{n,m,\beta}^{(i)}]^\vir
\end{align}

Define mixed type of ancestor and descendant  twisted correlator as follows:
\begin{align}
&\langle\langle\bar{\tau}_{j_1}\tau_{i_1}(\phi_1),\dots,\bar{\tau}_{j_n}\tau_{i_n}(\phi_n);c_{g,n}\rangle\rangle_g(\mathbf{t}(\psi))
\nonumber\\&:=\sum_{m,\beta\geq0}\frac{Q^\beta}{m!}
\int_{[\overline{\mathcal{M}}_{g,n+m}(X,\beta)]^\vir} \prod_{k=1}^{n}\left(\bar{\psi}_{k}^{j_k}\psi_k^{i_k}\ev_k^*\phi_k\right)\prod_{k=n+1}^{n+m}\ev_k^*\mathbf{t}(\psi_{k})\cdot \pi_m^* \St^*c_{g,n}.
\end{align}
for certain cohomology class $c_{g,n}\in H^*(\overline{\mathcal{M}}_{g,n})$, where $\psi_k:=c_1(\mathcal{L}_k)$ and $\bar{\psi}_k:=\pi_m^*\St^*c_1(L_k)$. 
Below we only focus the cases of $c_{g,n}=1$ or $\lambda_g$.

\begin{lemma}\label{lem:anc-desc}
For any $2g-2+n>0$,  on the big phase space, we have
\begin{align}\label{eqn:anc-des}
\langle\langle\bar\tau_{j_1}\tau_{i_1}(\phi_{\alpha_1}),\dots,\bar\tau_{j_n}\tau_{i_n}(\phi_{\alpha_n});\lambda_g\rangle\rangle_g =   \langle\langle  T^{j_1}(\tau_{i_1}(\phi_{\alpha_1})),\dots,  T^{j_n}(\tau_{i_n}(\phi_{\alpha_n}));\lambda_g\rangle\rangle_g 
\end{align}
for any $\{\phi_{\alpha_i}: i=1,\dots,n\}\subset H^*(X;\mathbb{C})$ and $T$ is the operator from \eqref{eq:T-operator}.
In particular,
\begin{align}\label{eqn:anc-des-2}
\langle\langle\bar\tau_{j_1}(\phi_{\alpha_1}),\dots,\bar\tau_{j_n}(\phi_{\alpha_n});\lambda_g\rangle\rangle_g =   \langle\langle  T^{j_1}(\phi_{\alpha_1}),\dots,  T^{j_n}(\phi_{\alpha_n});\lambda_g\rangle\rangle_g. 
\end{align}
Note that by \eqref{eqn:lambda-g-formula}, these two equations also hold true if we replace $\lambda_g$ by $\mathsf{P}_g^g(0,..,0)$. 
\end{lemma}
\begin{proof}
  We will focus on the first marked point, and for simplicity, suppress the content of the other marked
  points from our notation.
  By \eqref{eqn:D-Y-vir} and \eqref{eqn:prop-Hodge-1}, the following identity holds for mixed type twisted correlators on the big phase space
\begin{align*}
\langle\langle\bar{\tau}_{j_1-1}\tau_{i_1+1}(\phi),\dots;\lambda_g\rangle\rangle_g=
\langle\langle\bar{\tau}_{j_1}\tau_{i_1}(\phi),\dots;\lambda_g\rangle\rangle_g
+\langle\langle \tau_{i_1}(\phi),\phi_\beta\rangle\rangle_0\langle\langle\bar{\tau}_{j_1-1}(\phi^\beta),\dots;\lambda_g\rangle\rangle_g
\end{align*}
for any $\phi\in H^*(X;\mathbb{C})$.

This implies
\begin{align*}
&\langle\langle\bar{\tau}_{j_1}\tau_{i_1}(\phi_{\alpha_1}),\dots;\lambda_g\rangle\rangle_g 
=\langle\langle\bar\tau_{j_1-1}\left(\tau_{i_1+1}(\phi_{\alpha_1})-\langle\langle\tau_{i_1}(\phi_{\alpha_1}),\phi_\beta\rangle\rangle_0\phi^\beta\right),\dots;\lambda_g\rangle\rangle_g
\\=&\langle\langle\bar\tau_{j_1-1}T(\tau_{i_1}(\phi_{\alpha_1})),\dots;\lambda_g\rangle\rangle_g.
\end{align*}
Repeating this $j_1 - 1$ more times, we get
\begin{align*}
  \langle\langle\bar{\tau}_{j_1}\tau_{i_1}(\phi_{\alpha_1}),\dots;\lambda_g\rangle\rangle_g 
  =\langle\langle T^{j_1}(\tau_{i_1}(\phi_{\alpha_1})),\dots;\lambda_g\rangle\rangle_g
\end{align*}
Similar analysis for the other markings, yields \eqref{eqn:anc-des}.
\end{proof}

For any $g, n$ such that $2g - 2 + n > 0$ and $\mathbf{t}=\sum_{n,\alpha}t_n^\alpha z^n\phi_\alpha\in H^*(X)[[\{t_n^\alpha\}]][[z]]$, define a homomorphism $\Omega^{\mathbf{t}}_{g,n}\colon H^*(X)[[\psi]]^{\otimes n}\rightarrow H^*(\overline{\mathcal{M}}_{g,n})[[Q]][[\{t_n^\alpha\}]]$ via
\begin{align*}
\Omega^{\mathbf{t}}_{g,n}(\phi_{\alpha_1}\psi_1^{i_1},\dots,\phi_{\alpha_n}\psi_n^{i_n})  
=\sum_{m,\beta\geq0}\frac{Q^\beta}{m!}\St_*{\pi_m}_*\left(\prod_{k=1}^{n}\ev_k^*\phi_{\alpha_k}\psi_k^{i_k}\prod_{k=n+1}^{n+m}\ev_i^*\mathbf{t}(\psi_{k})\right).
\end{align*}
From the ``cutting edges'' axiom of virtual cycles of moduli of stable maps to $X$, this system of homomorphisms satisfies the splitting axioms:
\begin{align}\label{eqn:split-ax-1}
\iota_{g_1,g_2}^*\Omega_{g,n}^{\mathbf{t}}(\phi_{\alpha_1}\psi_1^{i_1},\dots,\phi_{\alpha_n}\psi_{n}^{i_n})
=\sum_{\beta}\Omega_{g_1,n_1+1}^{\mathbf{t}}(\otimes_{k\in I_1}\phi_{\alpha_k}\psi_k^{i_k},\phi_\beta)\otimes \Omega_{g_2,n_2+1}^{\mathbf{t}}(\otimes_{k\in I_2}\phi_{\alpha_k}\psi_k^{i_k}\otimes \phi^\beta)
\end{align}
and
\begin{align}\label{eqn:split-ax-2}
\iota_{g-1}^*\Omega_{g,n}^{\mathbf{t}}(\phi_{\alpha_1}\psi_1^{i_1},\dots,\phi_{\alpha_n}\psi_n^{i_n})
=\sum_{\beta}\Omega^{\mathbf{t}}_{g-1,n+2}(\phi_{\alpha_1}\psi_1^{i_1},\dots,\phi_{\alpha_n}\psi_n^{i_n},\phi_\beta,\phi^\beta)
\end{align}
where $\iota_{g_1,g_2}\colon \overline{\mathcal{M}}_{g_1,n_1+1}\times \overline{\mathcal{M}}_{g_2,n_2+1}\rightarrow \overline{\mathcal{M}}_{g,n}$ and $\iota_{g-1}\colon \overline{\mathcal{M}}_{g-1,n+2}\rightarrow \overline{\mathcal{M}}_{g,n}$ are the two canonical gluing maps.
Note that, by definition, $\Omega$ is related to the double bracket of descendant potential via integration on the moduli space of curve, 
i.e.
\begin{align*}
\int_{\overline{\mathcal{M}}_{g,n}}\Omega_{g,n}^{\mathbf{t}}(v_1\psi_1^{i_1},\dots,v_n\psi_{n}^{i_n})
=\langle\langle\tau_{i_1}(v_1),\dots,\tau_{i_n}(v_n)\rangle\rangle_g(\mathbf{t})
\end{align*}

Now let us prove Theorem~\ref{thm:deg-g-chern-chara}.
On the one hand, taking $l=g$ and $\mathbf{s}=0$ in \eqref{eqn:grr-FP-l>g} yields
\begin{align}
&\frac{(2g)!}{B_{2g}}
\langle\langle  \ch_{2g-1}(\mathbb{E})\rangle\rangle_{g}
\nonumber\\=&-\Tilde{t}_n^\alpha\langle\langle\tau_{n+2g-1}(\phi_\alpha)\rangle\rangle_g+\frac{1}{2}\sum_{i=0}^{2g-2}(-1)^{i}
\left\{\langle\langle\tau_i(\phi_\alpha)\tau_{2g-2-i}(\phi^\alpha)\rangle\rangle_{g-1}
+\sum_{h=0}^{g}\langle\langle\tau_{i}(\phi_\alpha)\rangle\rangle_h
\langle\langle\tau_{2g-2-i}(\phi^\alpha)\rangle\rangle_{g-h}\right\}
\end{align}
On the other hand, by combining \eqref{eqn:ch-2g-1-hodge} and Proposition~\ref{prop:lambda-g-lambda-g-1}, we obtain
\begin{footnotesize}
\begin{align*}
&\frac{(2g)!}{B_{2g}}
\ch_{2g-1}(\mathbb{E})
\nonumber\\=&
\frac{(-1)^{g}g}{2^{2g-2}}
\sum_{\Gamma\in G_{g,0}^\circ}\frac{1}{|\Aut(\Gamma)|}\sum_{\substack{\sum_{e\in E(\Gamma)}k(e)=g\\ k(e)\geq1, \forall e\in E(\Gamma)}}
(\xi_{\Gamma})_*\left(\prod_{e\in E(\Gamma)}\left(\frac{-1}{k(e)}\sum_{l(e)+m(e)=k(e)-1}\frac{\psi^{l(e)}_{h(e)}}{l(e)!}\frac{\psi_{h'(e)}^{m(e)}}{m(e)!}
 \right) \prod_{v\in V(\Gamma)}\mathsf{P}_{g(v)}^{g(v)}(0,0)\right).
 \end{align*}
\end{footnotesize}
Then combining with Lemma~\ref{lem:anc-desc}, the splitting
axioms~\eqref{eqn:split-ax-1} and \eqref{eqn:split-ax-2}, we get
\begin{small}
\begin{align*}
 &\frac{(2g)!}{B_{2g}}
\langle\langle  \ch_{2g-1}(\mathbb{E})\rangle\rangle_{g}
=\frac{(2g)!}{B_{2g}}
\int_{\overline{\mathcal{M}}_g}\Omega_{g,0}^{\mathbf{t}}\cdot \ch_{2g-1}(\mathbb{E})
\\=&\frac{(-1)^{g}g}{2^{2g-2}}
\sum_{\Gamma\in G_{g,0}^\circ}\frac{1}{|\Aut(\Gamma)|}\sum_{\substack{\sum_{e\in E(\Gamma)}k(e)=g\\ k(e)\geq1, \forall e\in E(\Gamma)}}
\int_{\overline{\mathcal{M}}_{\Gamma}}\xi_{\Gamma}^*\Omega_{g,0}^{\mathbf{t}}\cdot \Bigg(\prod_{e\in E(\Gamma)}\left(\frac{-1}{k(e)}\sum_{l(e)+m(e)=k(e)-1}\frac{\psi^{l(e)}_{h(e)}}{l(e)!}\frac{\psi_{h'(e)}^{m(e)}}{m(e)!}
 \right) 
 \\&\hspace{260pt}\cdot\prod_{v\in V(\Gamma)}\mathsf{P}_{g(v)}^{g(v)}(0,0)\Bigg)
% \\=&\frac{(2g)!}{B_{2g}}\frac{(-1)^{g-1}}{(2g-1)!}\frac{(-1)^{g}}{2^g}\sum_{m\geq1}\sum_{\Gamma\in G_{g,0}^{(m)}}B_{2g}\frac{1}{|Aut(\Gamma)|}\sum_{\substack{k_1+\dots+k_m=g\\ k_1,\dots,k_m\geq1}}
% \\&\prod_{i=1}^{m}\frac{-1}{2^{k_i}k_i!}\sum_{l_i=0}^{k_i-1}\binom{k_i-1}{l_i}
% \langle\langle T^{l_i}(\phi_{\alpha})T^{k_i-1-l_i}(\phi^\alpha);\lambda_{g(v_i)}\rangle\rangle_{g(v_i)}
% \\=&-\frac{g}{2^{g-1}}\sum_{m=1}^{g}\sum_{\Gamma\in G_{g,0}^{(m)}}\frac{1}{|Aut(\Gamma)|}\sum_{\substack{k_1+\dots+k_m=g\\ k_1,\dots,k_m\geq1}}
% \prod_{i=1}^{m}\frac{-1}{2^{k_i}k_i!}\sum_{l_i=0}^{k_i-1}\binom{k_i-1}{l_i}
% \langle\langle T^{l_i}(\phi_{\alpha})T^{k_i-1-l_i}(\phi^\alpha);\lambda_{g(v_i)}\rangle\rangle_{g(v_i)}
% \\=&-\frac{g}{2^{g-1}}\sum_{\substack{\Gamma\in G_{g,0}\\ h^1(\Gamma)=1}}\frac{1}{|Aut(\Gamma)|}\sum_{\substack{k_1+\dots+k_{E(\Gamma)}=g\\ k_1,\dots,k_{E(\Gamma)}\geq1}}
% \prod_{i=1}^{E(\Gamma)}\frac{-1}{2^{k_i}k_i!}\sum_{l_i=0}^{k_i-1}\binom{k_i-1}{l_i}
% \langle\langle T^{l_i}(\phi_{\alpha})T^{k_i-1-l_i}(\phi^\alpha);\lambda_{g(v_i)}\rangle\rangle_{g(v_i)}
% \\=&\frac{(-1)^{g}g}{2^{2g-2}}
% \sum_{\substack{\Gamma\in G_{g,0}\\ h^1(\Gamma)=1}}\frac{1}{|Aut(\Gamma)|}\sum_{\substack{k_1+\dots+k_{|V(\Gamma)|}=g\\ k_1,\dots,k_{|V(\Gamma)|}\geq1}}
% \prod_{i=1}^{|V(\Gamma)|}\frac{-1}{k_i!}\sum_{l_i=0}^{k_i-1}\binom{k_i-1}{l_i}
% \langle\langle T^{l_i}(\phi_{\alpha})T^{k_i-1-l_i}(\phi^\alpha);\mathsf{P}_{g(v_i)}^{g(v_i)}(0,0)\rangle\rangle_{g(v_i)}
\\=&\frac{(-1)^{g}g}{2^{2g-2}}
\sum_{m = 1}^{g-1} \frac{1}{2m}\sum_{\substack{k_1+\dots+k_m=g\\ k_1,\dots,k_m\geq1}} \sum_{\substack{g_1+\dots+g_m=g-1\\ g_1,\dots,g_m\geq1}}
\sum_{\alpha_{1},\dots,\alpha_{m}=1}^{N}\prod_{i=1}^{m}\frac{-1}{k_i!}\sum_{l_i=0}^{k_i-1}\binom{k_i-1}{l_i}
\langle\langle T^{l_i}(\phi_{\alpha_{i-1}})T^{k_i-1-l_i}(\phi^{\alpha_i});\mathsf{P}_{g_i}^{g_i}(0,0)\rangle\rangle_{g_i}
\end{align*}
\end{small}
where $\alpha_0=\alpha_m$.
In the last equality, we took into account the symmetries from stable circular graphs $\Gamma$ and indices $\{k_i\}_{i=1}^{m},\{g_i\}_{i=1}^{m}$.
Therefore, we prove Theorem~\ref{thm:deg-g-chern-chara}.

\subsection{Reformulation of Theorem \ref{thm:deg-g-chern-chara}}\label{subsec:Reformulation of Theorem 3}

In this subsection, explain how to use Pixton's formula to rewrite each  factor of the form $\langle\langle T^a(v_1)T^b(v_2);\mathsf{P}_{g}^{g}(0,0)\rangle\rangle_g$ in the right hand side of \eqref{eqn:LP-push-g} in terms of pure descendant Gromov-Witten invariants.
For convenience, let $\mathsf{P}^g_{\Gamma}(0,0)$ denote the constant term of the following polynomial in $r$ for sufficiently large $r$
\begin{align}\label{eqn:PgGamma00}
\mathsf{P}^g_{\Gamma}(0,0):=\text{Coeff}_{r^{0}}
\left[\sum_{w\in \mathsf{W}_{\Gamma,r}}
\frac{1}{|\Aut(\Gamma)| } 
\frac{1}{r^{h^1(\Gamma)}}
\prod_{e=(h,h')\in E(\Gamma)}
\frac{1-\exp(-w(h)w(h')(\psi_h+\psi_{h'}))}{\psi_h + \psi_{h'}}\right]
\end{align}
for any $\Gamma\in G_{g,2}$.
We may view $\mathsf{P}^g_{\Gamma}(0,0)$ as a polynomial in the
$\psi$-classes $\{\psi_h:h\in H(\Gamma)\}$.

We next define a linear operator $P \mapsto \langle\langle T^a(v_1)T^b(v_2); P\rangle\rangle_\Gamma$ that turns
any polynomial in the $\psi$-classes $\{\psi_h:h\in H(\Gamma)\}$ into a product of correlation functions.
Given a monomial $P = c \cdot \prod_{h\in H(\Gamma)} \psi_h^{a_h}$, we let
$\langle\langle T^a(v_1)T^b(v_2); P\rangle\rangle_\Gamma$ to be the contraction of the multilinear form
\begin{equation*}
  c \prod_{v\in V(\Gamma)} \langle\langle\ldots\rangle\rangle_{g(v)}
\end{equation*}
with arguments corresponding to half-edges $h \in H(\Gamma)$ via:
\begin{itemize}
\item for the argument corresponding to the first and second leg, use the vectors $T^a(v_1)$ and $T^b(v_2)$, respectively;
\item for each edge $e = (h,  h')$, use the bivector $T^{a_h}(\phi_\alpha)\otimes T^{a_{h'}}(\phi^\alpha)$ for the corresponding arguments.
\end{itemize}
\begin{proposition}\label{prop:TaTb}
\begin{equation*}
\langle\langle T^a(v_1)T^b(v_2);\mathsf{P}_{g}^{g}(0,0)\rangle\rangle_g
=\sum_{\Gamma\in {G}_{g,2}} \langle\langle T^a(v_1)T^b(v_2); \mathsf{P}^g_{\Gamma}(0,0)\rangle\rangle_\Gamma.
\end{equation*}
\end{proposition}

\begin{proof}
By \eqref{eqn:anc-des-2} in Lemma~\ref{lem:anc-desc} and \eqref{eqn:lambda-g-formula}, together with definition of $\Omega^{\mathbf{t}}$, we have
\begin{align*}
&\langle\langle T^a(v_1)T^b(v_2);\mathsf{P}_{g}^{g}(0,0)\rangle\rangle_g=\langle\langle \bar{\tau}_a(v_1)\bar{\tau}_b(v_2);\mathsf{P}_{g}^{g}(0,0)\rangle\rangle_g \\=&\sum_{m,\beta\geq0}\frac{Q^\beta}{m!}
\int_{[\overline{\mathcal{M}}_{g,m+2}(X,\beta)]^\vir} \left(\bar{\psi}_{1}^{a}\ev_1^* v_1\right)\left(\bar{\psi}_{2}^{b}\ev_2^* v_2\right)\prod_{k=3}^{m+2}\ev_k^*\mathbf{t}(\psi_{k})\cdot \pi_m^* \St^*\mathsf{P}_{g}^{g}(0,0)
\\=&\int_{\overline{\mathcal{M}}_{g,2}}\psi_1^a\psi_2^b\cdot \Omega^{\mathbf{t}}_{g,2}(v_1,v_2)\cdot \mathsf{P}_g^g(0,0). 
\end{align*}
By the projection formula, this equals
\begin{align*}
&\sum_{\Gamma\in {G}_{g,2}}\int_{\overline{\mathcal{M}}_{g,2}}\psi_1^a\psi_2^b\cdot\Omega^{\mathbf{t}}_{g,2}(v_1,v_2)\cdot {\xi_{\Gamma}}_*\mathsf{P}^g_{\Gamma}(0,0)
=\sum_{\Gamma\in {G}_{g,2}}\int_{\overline{\mathcal{M}}_{\Gamma}}\psi_1^a\psi_2^b\cdot {\xi_{\Gamma}}^*\Omega^{\mathbf{t}}_{g,2}(v_1,v_2)\cdot \mathsf{P}^g_{\Gamma}(0,0).
% \\=&\sum_{\Gamma\in {G}_{g,2}} 
% \sum_{w\in \mathsf{W}_{\Gamma,r}}
% \frac{1}{|\Aut(\Gamma)| } 
% \frac{1}{r^{h^1(\Gamma)}}
% \int_{\overline{\mathcal{M}}_{g,2}}\psi_1^a\psi_2^b\cdot \Omega^{\mathbf{t}}_{g,2}(\phi_\alpha,\phi^\alpha)\cdot \xi_{\Gamma*}\Bigg[ 
% \prod_{e=(h,h')\in E(\Gamma)}
% \frac{1-\exp(-w(h)w(h')(\psi_h+\psi_{h'}))}{\psi_h + \psi_{h'}} \Bigg]
% \nonumber\\=&\text{Coeff}_{r^0}\sum_{\Gamma\in {G}_{g,2}}\sum_{w\in \mathsf{W}_{\Gamma,r}}
% \frac{1}{|\Aut(\Gamma)| } 
% \frac{1}{r^{h^1(\Gamma)}} 
% \int_{\overline{\mathcal{M}}_{\Gamma}}\psi_1^a\psi_2^b\cdot {\xi_{\Gamma}}^*\Omega^{\mathbf{t}}_{g,2}(v_1,v_2)
% \\&\hspace{150pt}\cdot  
% \prod_{e=(h,h')\in E(\Gamma)}
% \frac{1-\exp(-w(h)w(h')(\psi_h+\psi_{h'}))}{\psi_h + \psi_{h'}}.
\end{align*}
Then by the splitting axioms~\eqref{eqn:split-ax-1}, \eqref{eqn:split-ax-2} and  equation~\eqref{eqn:PgGamma00}, we can express each summand as a product of descendant correlation functions according to the operator $P \mapsto \langle\langle T^a(v_1)T^b(v_2); P\rangle\rangle_\Gamma$.
In the process, the two marked points are assigned the vectors $v_1,v_2$, and the cotangent line class $\psi$ is translated to descendants using the operator $T$.
Each node is translated into a sum of pairs of primary vector fields $\phi_\alpha$ and $\phi^\alpha$.
\end{proof}
\subsection{Proof of Corollary~\ref{cor:LP-push-g}}
The following vanishing indentity for the Gromov--Witten inavriants of any smooth projective variety was first conjectured by  K.~Liu and H.~Xu in \cite{liu2009proof}, and was later proven in \cite{liu2011new}.
\begin{align}\label{eqn:grr-FP-l=g-LX}
0=-\Tilde{t}_n^\alpha\langle\langle\tau_{n+2g-1}(\phi_\alpha)\rangle\rangle_g+\frac{1}{2}\sum_{i=0}^{2g-2}(-1)^{i}
\left\{\sum_{h=0}^{g}\langle\langle\tau_{i}(\phi_\alpha)\rangle\rangle_h
\langle\langle\tau_{2g-2-i}(\phi^\alpha)\rangle\rangle_{g-h}\right\}
\end{align}
for $n\geq0$.

Combining \eqref{eqn:deg-g-chern-chara} and \eqref{eqn:grr-FP-l=g-LX}, we obtain
\begin{small}
\begin{align}\label{eqn:push-for-tau}
 &\sum_{i=0}^{2g-2}(-1)^{i}\langle\langle\tau_{i}(\phi_\alpha)\tau_{2g-2-i}(\phi^\alpha)\rangle\rangle_{g-1} \nonumber 
\\=&\frac{(-1)^{g}g}{2^{2g-2}}
\sum_{m=1}^{g-1} \frac{1}{m}\sum_{\substack{k_1+\dots+k_m=g\\ k_1,\dots,k_m\geq1}} \sum_{\substack{g_1+\dots+g_m=g-1\\ g_1,\dots,g_m\geq1}}
\sum_{\alpha_1,\dots,\alpha_m=1}^{N}\prod_{i=1}^{m}\frac{-1}{k_i!}\sum_{l_i=0}^{k_i-1}\binom{k_i-1}{l_i}
\langle\langle T^{l_i}(\phi_{\alpha_{i-1}})T^{k_i-1-l_i}(\phi^{\alpha_i});\mathsf{P}_{g_i}^{g_i}(0,0)\rangle\rangle_{g_i}.
\end{align}
\end{small}
Recall that for any arbitrary contravariant tensors $P$ and $Q$ on the big phase space, the following formula
was proved in \cite[Proposition 3.2]{liu2011certain}:
\begin{align*}
    \sum_{i=0}^{m}(-1)^{i}P(\tau_i(\phi_\alpha))Q(\tau_{m-i}(\phi^\alpha))=\sum_{i=0}^{m}(-1)^{i}P(T^i(\phi_\alpha))Q(T^{m-i}(\phi^\alpha))
\end{align*}
for $m\geq0$. 
In particular, if we take $P(W)=Q(W)=W$, 
the left hand side of \eqref{eqn:push-for-tau} becomes
\[\sum_{i=0}^{2g-2}(-1)^{i}\langle\langle T^{i}(\phi_\alpha)T^{2g-2-i}(\phi^\alpha)\rangle\rangle_{g-1}.\]
This completes the proof of Corollary~\ref{cor:LP-push-g}.  

\section{Push-forward relations on $\overline{\mathcal{M}}_{g}$ for $(g\geq2)$}\label{sec:push-forward relation}

\subsection{Push-forward relations}
First, we introduce our  definition of push-forward relations, which is slightly different from the notion proposed in \cite[Section~4.9.2]{pandharipande2015calculus}.  Consider the gluing map,
\[\iota\colon \overline{\mathcal{M}}_{g,2}\rightarrow \overline{\mathcal{M}}_{g+1}\]
with image equal to the boundary divisor $\Delta_0\subset \overline{\mathcal{M}}_{g+1}$ of curves with a nonseparating node. 
A push-forward relation can be understand as an  element of $\ker(\iota_*)$,
where $\iota_*\colon R^*(\overline{\mathcal{M}}_{g,2})\rightarrow R^*(\overline{\mathcal{M}}_{g+1})$. 
Push-forward relations among descendant stratum classes  yield universal equations for genus less than or equal to $g$   Gromov--Witten invariants, despite the fact that the relations are embedded in  genus $g+1$ moduli space of curves.
A simple family of push-forward relations in $R^{2g+r+1}(\overline{\mathcal{M}}_{g+1})$ for $(g\geq1, r\geq1)$ were first found in \cite{liu2011new}
\begin{align}\label{eqn:push-pand-liu}
\sum_{a+b=2g+r}(-1)^{a}\iota_*\left({\psi_1}^{a}{\psi_2}^{b}\right)=0.    
\end{align}
Comparing equations~\eqref{eqn:LP-push-g-class} and \eqref{eqn:push-pand-liu}, we see \eqref{eqn:LP-push-g-class} gives the correction terms explicitly for the left hand side of equation~\eqref{eqn:push-pand-liu} for $r=0$. 

\subsection{Proof of Theorem~\ref{thm:LP-push-g-class}}
Let $j\colon\Delta(g_1,g_2)\cong\overline{\mathcal{M}}_{g_1,2}\times \overline{\mathcal{M}}_{g_2,2}\rightarrow\overline{\mathcal{M}}_{g+1,2}$ denote the boundary divisor parametrizing nodal curves with two components $C_1$ and $C_2$ with genus $g_1$ and $g_2$, where $g_1+g_2=g+1$.
Let $\psi_{*i}$ be the cotangent line class at the node along the curve $C_i$ for $i=1,2$.  
By the relation \cite[Proposition 2]{liu2011new} in genus $g+1$, we have
\begin{align}\label{eqn:LP-prop2}
-\psi_1^{2g+3}-\psi_2^{2g+3}+\sum_{g_1+g_2=g+1: g_i>0}\sum_{a+b=2g+2}(-1)^{a}j_*\left(\psi_{*1}^a\psi_{*2}^b\cap [\Delta(g_1,g_2)]\right)=0    
\end{align}
in $R^{2g+3}(\overline{\mathcal{M}}_{g+1,2})$.
Pushing forward equation~\eqref{eqn:LP-prop2} to $\overline{\mathcal{M}}_{g+1}$ by forgetting the two marked points,  via string equation,  we have the following relation on $\overline{\mathcal{M}}_{g+1}$
\begin{align}
\kappa_{2g+1}+\frac{1}{2}\sum_{g_1+g_2=g,g_i>0}\sum_{a+b=2g}(-1)^{a}j_*\left(\psi_{*1}^a\psi_{*2}^b\cap [\Delta(g_1,g_2)]\right)=0.   
\end{align}
Meanwhile, by Mumford’s Grothendieck--Riemann--Roch calculation (cf. \cite{mumford1983towards}), the following holds on $\overline{\mathcal{M}}_{g+1}$:
\begin{align*}
  &\frac{(2g+2)!}{B_{2g+2}}\ch_{2g+1}(\mathbb{E}) \\
  =&\kappa_{2g+1}+\frac{1}{2}\iota_*\left(\sum_{a+b=2g}(-1)^{a}\psi_{1}^{a}\psi_{2}^{b}\right)+\frac{1}{2}\sum_{g_1+g_2=g,g_i>0}\sum_{a+b=2g}(-1)^{a}j_*\left(\psi_{*1}^a\psi_{*2}^b\cap [\Delta(g_1,g_2)]\right)
\end{align*}
Thus  we get the following relation on $\overline{\mathcal{M}}_{g+1}$,
\begin{align*}
 \frac{(2g+2)!}{B_{2g+2}}\ch_{2g+1}(\mathbb{E}) =\frac{1}{2}\iota_*\left(\sum_{a+b=2g}(-1)^{a}\psi_{1}^{a}\psi_{2}^{b}\right)\end{align*}
Together with equation~\eqref{eqn:ch-2g-1-hodge} and Proposition~\ref{prop:lambda-g-lambda-g-1}, we obtain Theorem~\ref{thm:LP-push-g-class}.

\subsection{Connection to degree-$g$ topological recursion relations}\label{subsec:relation-hodge-trr}
% In $g=2$ case, 
% \begin{align*}
%  &\sum_{i=0}^{2}(-1)^{i}\langle\langle\tau_{i}(\phi_\alpha)\tau_{2-i}(\phi^\alpha)\rangle\rangle_{1}  
%  =\frac{(4)!}{B_{4}}\frac{(-1)^{1}}{(3)!}
% \langle\langle \lambda_2\lambda_{1}\rangle\rangle_{2}
% ==\frac{(4)!}{B_{4}}\frac{(-1)^{1}}{(3)!}\frac{1}{240}
% \langle\langle T(\phi_\alpha)\phi^\alpha;\lambda_{1}\rangle\rangle_{1}
% \end{align*}
% we need a formula for $\lambda_1$ on $\overline{\mathcal{M}}_{1,2}$.
We propose the following connection between the relations \eqref{eqn:LP-push-g}, and a class of degree-$g$ topological recursion relations on the moduli space of stable curves:
\begin{conjecture}
  The universal equation \eqref{eqn:LP-push-g} is a consequence of the degree-$g$ topological recursion relations proposed in \cite{clader2022topological}.
\end{conjecture}
We verify the conjecture for $g=2$ case in the example below.
\begin{example}
  For genus $g=2$, the equation \eqref{eqn:LP-push-g} specializes to
\begin{align}\label{eqn:LP-push-g=2}
 &\sum_{i=0}^{2}(-1)^{i}\langle\langle T^{i}(\phi_\alpha)T^{2-i}(\phi^\alpha)\rangle\rangle_{1}  \nonumber
% \\=&\frac{(-1)^{g}g}{2^{2g-2}}
% \sum_{\substack{\Gamma\in G_{g,0}\\ h^1(\Gamma)=1}}\frac{1}{|\Aut(\Gamma)|}\sum_{\substack{k_1+\dots+k_{E(\Gamma)}=g\\ k_1,\dots,k_{E(\Gamma)}\geq1}}
% \prod_{i=1}^{E(\Gamma)}\frac{-1}{k_i!}\sum_{l_i=0}^{k_i-1}\binom{k_i-1}{l_i}
% \langle\langle T^{l_i}(\phi_{\alpha})T^{k_i-1-l_i}(\phi^\alpha);\mathsf{P}_{g(v_i),2}^{g(v_i)}(0,0)\rangle\rangle_{g(v_i)}    
\\=&\frac{2}{2^2}
\sum_{m=1}^1 \frac{1}{m}\sum_{\substack{k_1+\dots+k_m=2\\ k_1,\dots,k_m\geq1}} \sum_{\substack{g_1+\dots+g_m=1\\ g_1,\dots,g_m\geq1}}
\prod_{i=1}^{m}\frac{-1}{k_i!}\sum_{l_i=0}^{k_i-1}\binom{k_i-1}{l_i}
\langle\langle T^{l_i}(\phi_{\alpha_{i-1}})T^{k_i-1-l_i}(\phi^{\alpha_i});\mathsf{P}_{g_i}^{g_i}(0,0)\rangle\rangle_{g_i}
\end{align}
We will check that \eqref{eqn:LP-push-g=2} is a consequence of the genus-1 topological recursion relation:
\begin{align}\label{eqn:genus-1-TRR}
  \langle\langle T(W)\rangle\rangle_1 =\frac{1}{24}
  \sum_{\alpha}\langle\langle W\phi_\alpha\phi^\alpha\rangle\rangle_0
\end{align}
for any vector field $W$ on the big phase space. 
We start by rewriting the left hand side of \eqref{eqn:LP-push-g=2} as
\begin{align*}
&\sum_{\alpha}\sum_{i=0}^{2}(-1)^{i}\langle\langle T^{i}(\phi_\alpha)T^{2-i}(\phi^\alpha)\rangle\rangle_{1}
=2\sum_{\alpha}\langle\langle \phi_\alpha T^2(\phi^\alpha)\rangle\rangle_1-\sum_{\alpha}\langle\langle T(\phi_\alpha) T(\phi^\alpha)\rangle\rangle_1
\\=&\frac{1}{24}\sum_{\alpha,\beta}\langle\langle \phi_\alpha T(\phi^\alpha)\phi_\beta\phi^\beta\rangle\rangle_0,
\end{align*}
where we have used \eqref{eqn:genus-1-TRR} and its derivatives in the last equality.
The right hand side of \eqref{eqn:LP-push-g=2} equals
\begin{align*}
&-\frac{1}{4}\sum_{\alpha}\sum_{l_1=0}^{1}\binom{1}{l_1}
\langle\langle T^{l_1}(\phi_{\alpha})T^{1-l_1}(\phi^{\alpha});\mathsf{P}_{1}^{1}(0,0)\rangle\rangle_{1}
\\=&-\frac{1}{4}\sum_{\alpha}
\langle\langle T(\phi_{\alpha})\phi^\alpha;\mathsf{P}_{1}^{1}(0,0)\rangle\rangle_{1}
-\frac{1}{4}\sum_{\alpha}
\langle\langle \phi_{\alpha}T(\phi^\alpha);\mathsf{P}_{1}^{1}(0,0)\rangle\rangle_{1}
\\=&
-\frac{1}{2}\sum_{\alpha}
\langle\langle \phi_{\alpha}T(\phi^\alpha);\mathsf{P}_{1}^{1}(0,0)\rangle\rangle_{1}
=\frac{1}{24}\sum_{\alpha,\beta}\langle\langle\phi_\alpha T(\phi^\alpha)\phi_\beta\phi^\beta \rangle\rangle_0
\end{align*}
where we have used the identity
\begin{align*}
\sum_{\alpha}
\langle\langle T(\phi_{\alpha})\phi^\alpha;\mathsf{P}_{1}^{1}(0,0)\rangle\rangle_{1}=\sum_{\alpha,\beta}\eta^{\alpha\beta}
\langle\langle T(\phi_{\alpha})\phi_\beta;\mathsf{P}_{1}^{1}(0,0)\rangle\rangle_{1}
=\sum_{\beta}
\langle\langle \phi_{\beta}T(\phi^\beta);\mathsf{P}_{1}^{1}(0,0)\rangle\rangle_{1}. 
\end{align*}
in the second equality, and the formula~\eqref{eqn:P_1^1(00)} for $\mathsf{P}_1^1(0,0)$ in the last equality.
\end{example}

\section{Chern class of Hodge bundle}\label{sec:Chern class of Hodge bundle}

Unlike the Chern characters, the Chern classes $\lambda_i=c_i(\mathbb{E})$ already vanish in degrees $i>g$.
Thus for degree reasons, we believe the vanishing of the $\lambda_i$ implies much stronger relations than the relations from the Chern characters we have studied in Section~\ref{sec:chern-character-Hodge}, and this makes it very interesting to study what kind of relations the vanishing of the $\lambda_i$ for $i > g$ leads to.

\subsection{Linear Hodge integrals}

Define the linear Hodge integral
\begin{align}
\langle\tau_{k_1}(\phi_{\alpha_1}),\dots,\tau_{k_n}(\phi_{\alpha_n});\lambda_j\rangle_{g,n,\beta}:=\int_{[\overline{\mathcal{M}}_{g,n}(X,\beta)]^\vir}\prod_{i=1}^{n}\psi_i^{k_i} \ev_i^*\phi_{\alpha_i}\cdot \lambda_j    
\end{align} 
where $\lambda_j=c_j(\mathbb{E})$  is the $j$-th Chern class of the Hodge bundle $\mathbb{E}$.

The generating function of genus-$g$ linear
Hodge integrals is given by
\begin{align}\label{eqn:def-linear-Hodge}
F_g^{\mathbb{E}}(\mathbf{t},u):=\Bigg\langle\exp\Big(\sum_{n,\alpha}t_n^\alpha\tau_{n}(\phi_\alpha)\Big);\lambda(u)\Bigg\rangle_{g}
\end{align}
where $\lambda(u)=\sum_{i=0}^{g}\lambda_i u^i$ is the Chern polynomial of the Hodge bundle $\mathbb{E}$, and $u$ is a formal parameter.
Via a standard identity relating the Chern classes and Chern characters of a vector bundle, we have
\begin{align}
\lambda(u)=\exp\left(\sum_{l\geq1}(2l-2)!u^{2l-1}\ch_{2l-1}(\mathbb{E})\right) .   
\end{align}
The generating function of genus-$g$ linear Hodge integrals $F^{\mathbb{E}}_{g}(\mathbf{t},u)$ can be obtained from $F_g^{\mathbb{E}}(\mathbf{t},\mathbf{s})$ via the substitution
\begin{align}\label{eqn:tran-s-u}
s_{2i-1}=(2i-2)!\cdot u^{2i-1}.
\end{align}

\subsection{The Faber--Pandharipande relation and Givental quantization}
\subsubsection{The Faber--Pandharipande relation}
Define the differential operator for $l\geq 1$
\begin{equation}
  \label{eq:D-operator}
  D_{2l-1}:=-\sum_{n=0}^{\infty}\sum_{\alpha}\tilde{t}_n^\alpha\frac{\partial}{\partial t^\alpha_{n+2l-1}}+\frac{\hbar}{2}\sum_{n=0}^{2l-2}(-1)^{n}\eta^{\alpha\beta}\frac{\partial}{\partial t_n^\alpha}\frac{\partial}{\partial t_{2l-2-n}^\beta}
\end{equation}
% \begin{align*}
% &Z_{\mathbb{E}}^X=\exp(\sum_{l\geq1}\frac{B_{2l}}{(2l)!}s_{2l-1}D_{2l-1})Z^X    
% \\=&\exp(\hat{r}_{leg}+\hat{r}_{tail}+\hat{r}_{edge})Z^X
% \end{align*}
Let $\mathcal{D}^{\mathbb{E}}(\mathbf{t},\mathbf{s})=\exp\left(\sum_{g\geq0}\hbar^{g-1}F^{\mathbb{E}}_g(\mathbf{t},\mathbf{s})\right)$ be the total descendant potential of Hodge integrals.
The following differential equation was proved in \cite[Propositon 2]{faber2000hodge}
\begin{align}\label{eqn:FP-prop2}
\frac{\partial}{\partial s_{2l-1}}\mathcal{D}^{\mathbb{E}}(\mathbf{t},\mathbf{s})=\frac{B_{2l}}{(2l)!}D_{2l-1}\cdot \mathcal{D}^{\mathbb{E}}(\mathbf{t},\mathbf{s})   
\end{align}
for any $l\geq1$. 
The identity \eqref{eqn:FP-prop2} may be used to derive the following formula (cf. \cite{givental2001gromov}) 
\begin{align}\label{eqn:Zhodge-Z-ord}
 \mathcal{D}^{\mathbb{E}}(\mathbf{t},\mathbf{s})
 =\exp(\sum_{l\geq1}\frac{B_{2l}}{(2l)!}s_{2l-1}D_{2l-1})\cdot   \mathcal{D}^{X}(\mathbf{t})
\end{align}
for $\mathcal{D}^{\mathbb{E}}(\mathbf{t},\mathbf{s})$ in terms of $\mathcal{D}^{X}(\mathbf{t})=\exp\left(\sum_{g\geq0}\hbar^{g-1}F_g(\mathbf{t})\right)$, the total descendant potential of Gromov--Witten invariants of $X$.
\subsubsection{Quantization of quadratic Hamiltonians}

In this section, we give a brief review of Givental's quantization of quadratic Hamiltonians (see \cite{givental2001gromov} for more details).

Let $z$ be a formal variable. We consider the space $\mathbb{H}$ which is the space of Laurent polynomials in one variable $z$ with coefficients in $H^*(X;\mathbb{C})$.
We define a symplectic form $\Omega$ on $\mathbb{H}$ via
\[\Omega(f,g)= \operatorname{Res}_{z=0}\eta(f(-z),g(z))dz\]
for any $f,g\in\mathbb{H}$. Note that we have $\Omega(f,g)=-\Omega(g,f)$. There is a natural polarization $\mathbb{H}=\mathbb{H}_+\oplus \mathbb{H}_-$ corresponding to the decomposition $f(z,z^{-1})=f_+(z)+f_-(z^{-1})z^{-1}$ of Laurent polynomials into polynomial and polar parts. It is easy to see that $\mathbb{H}_+$ and $\mathbb{H}_-$ are both Lagrangian subspaces of $\mathbb{H}$ with respect to $\Omega$.

Introduce a Darboux coordinate system $\{p^\beta_m,q^\alpha_n\}$ on $\mathbb{H}$ with respect to the above polarization. A general element $f\in\mathbb{H}$ can be written in the form
\[\sum_{m\geq 0,\beta}p^\beta_m\phi^\beta(-z)^{-m-1}+\sum_{n\geq 0,\alpha}q^\alpha_n\phi_{\alpha} z^n,\]
where $\{{\phi}^\alpha\}$ is the dual basis to $\{\phi_\alpha\}$.
Denote
\begin{eqnarray*}
\mathbf{p}(z):=\sum_{m\geq 0,\beta}p^\beta_m\phi^\beta(-z)^{-m-1},\quad
\mathbf{q}(z):=\sum_{n\geq 0,\alpha}q^\alpha_n\phi_{\alpha} z^n.
\end{eqnarray*}
For convenience, we identify $\mathbb{H}_{+}$ with the big phase space via identification between basis $\{z^{n}\phi_\alpha\}$ and $\{\tau_n(\phi_\alpha)\}$, and dilaton shift about coordinates $q_n^\alpha=t_n^\alpha-\delta_{n}^{1}\delta_{\alpha}^{1}$.

Let $A\colon\mathbb{H}\to\mathbb{H}$ be a linear infinitesimal symplectic transformation, i.e.\ $\Omega(Af,g)+\Omega(f,Ag)=0$ for any $f,g\in\mathbb{H}$.
In Darboux coordinates, the quadratic Hamiltonian
\[f\to\frac{1}{2}\Omega(Af,f)\]
is a series of homogeneous degree two monomials in $\{p^\beta_m,q^\alpha_n\}$.

We define the quantization of quadratic monomials as
\[
  \widehat{q^\beta_mq_n^\alpha}=\frac{q^\beta_m q_n^\alpha}{\hbar},\quad \widehat{q^\beta_m p_n^\alpha}=q^\beta_m\frac{\partial}{\partial q^\alpha_n},\quad
  \widehat{p^\beta_m p_n^\alpha}=\hbar \frac{\partial}{\partial q^\beta_m}\frac{\partial}{\partial q^\alpha_n},
\]
where $\hbar$ is a formal variable.
We define the quantization $\widehat{A}$ of a general $A$ by extending the above equalities linearly.
The differential operators $\widehat{q^\beta_mq_n^\alpha},\widehat{q^\beta_mp_n^\alpha},\widehat{p^\beta_mp_n^\alpha}$ act on the \emph{Fock space}, which is the space of formal functions in $\mathbf{q}(z)\in\mathbb{H}_+$.
For example, the descendant potential $\mathcal{D}(\mathbf{t})$ may be regarded as an element in the Fock space via the dilaton shift $\mathbf{q}(z)=\mathbf{t}(z)-\mathbf{1} z$.
The quantization of a symplectic transform of the form $\exp(A)$, with $A$ infinitesimal symplectic, is defined as $\exp(\widehat{A})=\sum_{n\geq 0}\frac{\widehat{A}^n}{n!}$.

\subsection{Universal relations from the Chern classes of the Hodge bundle}

It is easy to see that multiplication by $z^{2k-1}$ defines an infinitesimal symplectic transformation on $(\mathbb{H},\Omega)$, and we denote by $\widehat{z^{2k-1}}$ the corresponding quantization.
It may be checked that $\widehat{z^{2k-1}}$ coincides with the operator $D_{2k-1}$ from \eqref{eq:D-operator}.
For convenience, we will use the following convention
\begin{align*}
\widetilde{\mathcal{P}}(\mathbf{q})=\mathcal{P}(\mathbf{t})|_{\mathbf{t}(z)=\mathbf{q}(z)+\mathbf{1}z}    
\end{align*}
for any partition function $\mathcal{P}(\mathbf{t})$.
This allows us to restate \eqref{eqn:Zhodge-Z-ord} as
\begin{align}\label{eqn:quan-ZE-Z}
\widetilde{\mathcal{D}}^{\mathbb{E}}(\mathbf{q},\mathbf{s})=\widehat{R(z)}\cdot \widetilde{\mathcal{D}}^{X}(\mathbf{q})
\end{align}
for the $R$-matrix
\begin{equation*}
  R(z)=\exp(\sum_{i=1}^{\infty}\frac{B_{2i}}{(2i)!}s_{2i-1}z^{2i-1}).
\end{equation*}
So under the substitution \eqref{eqn:tran-s-u}, taking the logarithm of equation~\eqref{eqn:quan-ZE-Z}, 
we obtain
\begin{equation*}
  \sum_{g\geq0}\hbar^{g-1}\widetilde{F}_{g}^{\mathbb{E}}(\mathbf{q},u)= \log\left(\exp\Big(\sum_{l\geq1}\frac{B_{2l}}{(2l)(2l-1)}u^{2l-1}\widehat{z^{2l-1}}\Big)\widetilde{\mathcal{D}}^X(\mathbf{q})\right). 
\end{equation*}
Denote 
\begin{align}
&\langle\langle\tau_{i_1}(\phi_1),\dots,\tau_{i_n}(\phi_n);c_{g,n}\rangle\rangle^{\sim}_g(\mathbf{q}(z))=\langle\langle\tau_{i_1}(\phi_1),\dots,\tau_{i_n}(\phi_n);c_{g,n}\rangle\rangle_g(\mathbf{t}(z))|_{\mathbf{t}(z)=\mathbf{q}(z)+1z}
\end{align}
for certain cohomology class $c_{g,n}\in H^*(\overline{\mathcal{M}}_{g,n})$. 
\begin{proposition}
  Let $g > 0$ and $i \ge 0$.
  Then, the following holds for descendant Gromov--Witten invariants
  \begin{equation*}
    [u^i \hbar^{g-1}] \log\left(\exp\Big(\sum_{l\geq1}\frac{B_{2l}}{(2l)(2l-1)}u^{2l-1}\widehat{z^{2l-1}}\Big)\widetilde{\mathcal{D}}^X(\mathbf{q})\right)
    =
    \begin{cases}
      0 & \text{, if $i > g$}, \\
      \frac{(-1)^g}{2^g}\cdot\langle\langle \,\,;\mathsf{P}_{g}^g(\emptyset)\rangle\rangle_g^{\sim}(\mathbf{q}(z)) & \text{, if $i=g$}.
    \end{cases}
  \end{equation*}
Here $[u^i\hbar^{g-1}]$ denotes taking the coefficient of $u^i\hbar^{g-1}$. 
\end{proposition}
\begin{proof}
  This follows from the vanishing $\lambda_i = 0$ if $i > g$, and equation~\eqref{eqn:lambda-g-formula}.
\end{proof}

\begin{theorem}[\cite{givental2004symplectic}]\label{thm:givental-quant}
  Let $\mathbf{q}(z)=\sum_{m,\alpha}q_m^\alpha\phi_\alpha z^m$ be an element of $\mathbb{H}_+$.
  Given a partition function $\mathcal{D}(\mathbf{q})$ on $\mathbb{H}_+$, the quantized operators act as follows.
  The quantization of an $R$-matrix $R$ acts on $\mathcal D$ via
\[
\widehat{R} \mathcal{D}(\mathbf{q})=\left[e^{\frac{\hbar}{2} V\left(\frac{\partial}{\partial_{\mathbf{q}}},\frac{\partial}{\partial_{\mathbf{q}}}\right)}\mathcal{D}\right](R^{-1}\mathbf{q})
\]
where $R^{-1}\mathbf{q}$ is the power series $R^{-1}(z)\mathbf{q}(z)$, and the quadratic form $V=\sum_{k,l}(p_k,V_{kl}p_l)$ is defined via
\[
V(w,z)=\sum_{k,l\geq 0}V_{kl}w^kz^l=\frac{1-R(-w)^*R(-z)}{w+z},
\]
where $*$ denotes the adjoint with respect to the intersection paring on $H^*(X;\mathbb{C})$.
\end{theorem}

The action of the exponential of the quadratic differential operator in Theorem~\ref{thm:givental-quant} has a Feynman
graph expansion, and its logarithm  only involves connected graphs (c.f. \cite{clader2018geometric}).
Let $\Gamma$ denote a
connected dual graph consisting of vertices $V$, edges $E$, and legs $L$, with each vertex $v$ labeled by
a genus $g(v)$.
We define the genus of the graph by $g(\Gamma)=h^1(\Gamma)+\sum_{v\in V}g(v)$
where $h^1$ denotes the first Betti number of the Feynman graph $\Gamma$.
Denote $G_{g}^{\mathrm{Feyn}}$ to be the set of connected Feynman graphs of genus $g$.
Feynman graphs differ from the stable graphs from Section~\ref{pixconj} in that they allow genus zero vertices of valence $1$ or $2$.
\begin{theorem} \label{thm:Hodge-class-relation}
  Let
  \begin{equation*}
    R(z)=\exp(\sum_{i=1}^{\infty}\frac{B_{2i}}{(2i)(2i-1)}u^{2i-1}z^{2i-1}).
  \end{equation*}
  Then, for $g\geq2$, the following holds for descendant Gromov--Witten invariants
  \begin{align}
  \label{eqn:Hodge-class-relation}[u^i]\sum_{\Gamma\in G^{\mathrm{Feyn}}_{g}}\frac{1}{|\Aut(\Gamma)|}\Cont_{\Gamma}
  =\begin{cases}
0, &i>g
\\
\frac{(-1)^g}{2^g}\cdot\langle\langle \,\,;\mathsf{P}_{g}^g(\emptyset)\rangle\rangle_g^{\sim}(\mathbf{q}(z))
&i=g
\end{cases}
 \end{align}
where $[u^i]$ denotes taking the coefficient of $u^i$ and  the contribution $\Cont_{\Gamma}$ of a Feynman graph $\Gamma$ is defined as a contraction of tensors:
\begin{itemize}
\item{Each vertex $v$ is assigned a tensor bracket
$$\langle\langle\quad\rangle\rangle_{g(v)}^{\sim}(R^{-1}\mathbf{q})$$
with insertion coming from the half edges.}
\item Each edge $e$ is assigned a bivector with descendants
    \begin{equation*}
      \frac{\sum_{\alpha,\beta}\left(\eta^{\alpha\beta}\phi_{\alpha}\otimes\phi_{\beta}-\eta^{\alpha\beta} R^{-1}(\psi_{h(e)})(\phi_{\alpha})\otimes R^{-1}(\psi_{h'(e)})(\phi_{\beta})\right)
}{\psi_{h(e)}+\psi_{h'(e)}}
    \end{equation*}
\end{itemize}
\end{theorem}

\subsection{Connection to Pixton's 3-spin relations}

In 2011, Pixton \cite{pixton2012conjectural} proposed a large class of
conjectural relations in the tautological ring of
$\overline{\mathcal{M}}_{g,n}$, which extend the Faber--Zagier
relations on the moduli space of smooth curves $\mathcal{M}_g$.
In \cite{pandharipande2015relations}, these relations were proven to
hold in cohomology by studying the cohomological field theory of
Witten's 3-spin class, and so these relations are now usually called
\emph{Pixton's 3-spin relations}.
In \cite{Ja17}, the relations were established in Chow via a study of
the equivariant Gromov--Witten theory of $\PP^1$.
Pulling back to the moduli space of stable maps, Pixton's 3-spin
relations lead to a large class of differential equations for ancestor
Gromov--Witten invariants of any target manifold.

This work leads to a natural question:
\begin{question}
  Is there any connection between the Hodge class relations
  \eqref{eqn:Hodge-class-relation} and Pixton's 3-spin relations?
\end{question}

\begin{comment}
To study this question, recall the definition of the total ancestor
potential of $X$
\begin{align*}
  \mathcal{A}_{\tau}(\mathbf{t})=\exp\left(\sum_{g\geq0}\hbar^{2g-2}\bar{F}_{g}(\tau,\mathbf{t})\right),
\end{align*}
where $\bar{F}_g$ is the genus-g ancestor potential
\begin{align*}
\bar{F}_g(\tau,\mathbf{t})
=\sum_{\beta\in H_2(X)}\sum_{m,n=0}^{\infty}\frac{Q^\beta}{m!n!}
\int_{\overline{\mathcal{M}}_{g,m+n}(X,\beta)}\prod_{i=1}^{m}\ev_i^*\mathbf{t}(\bar{\psi}_i)\prod_{i=m+1}^{m+n}\ev_i^*\tau
\end{align*}
for some element $\tau\in H^*(X)$ called the \emph{basepoint}.
We will also use the $S$-operator defined in \cite{givental2001gromov} via
\begin{align*}
S_\tau(z)(v)=v+\sum_{\alpha}\langle\langle\frac{v}{z-\psi},\phi^\alpha\rangle\rangle_0(\tau)\phi_\alpha    
\end{align*}
for $v\in H^*(X)$.
This is also a symplectic operator, and its quantization plays an important role in the descendant/ancestor
correspondence (cf. \cite{coates2007quantum}):
\begin{align}\label{eqn:dec-anc-corr}
\mathcal{D}= e^{F_1(\tau)}\widehat{S(\tau)^{-1}}\mathcal{A}_\tau \end{align}
where $F_1(\tau)$ is the genus-1 primary potential.
The total ancestor potential $\mathcal D^\EE$ and the $S$-operator $S^\EE$ can be similarly defined for linear Hodge integrals.
Note that the $S$-operators are defined via genus-0 two-point invariants, and so the two operators $S_\tau(z)$ and $S_\tau^{\mathbb{E}}(z)$ agree.
\end{comment}

To answer this question, we start by making a connection between
Pixton's 3-spin relations, and the Chern classes of the Hodge bundle.
There is a canonical way of writing the $\lambda_i$ in terms of
boundary classes, $\psi$-classes and $\kappa$-classes, i.e.\ as
elements of the strata algebra.
As explained in \cite[\S1.3]{pandharipande2015calculus}, the Chern
polynomial $\lambda(u)$ forms a one-dimensional semisimple CohFT
$\Omega^{\lambda(u)}$.
The Givental--Teleman classification for $\Omega^{\lambda(u)}$ yields
the required canonical expression for $\lambda(u)$.

\begin{proposition}
  \label{prop:3-spin}
  Pixton's 3-spin relations on the moduli space of stable curves imply
  the vanishing $\lambda_i = 0$ for $i > g$.
  More precisely, as elements of the strata algebra, we may write
  $\lambda_i$ as a linear combination of Pixton's 3-spin relations.
\end{proposition}
\begin{proof}
  We use the results of \cite{janda2018frobenius}, which show that a
  large class of tautological relations from cohomological field
  theories are linear combinations of the 3-spin relations.
  In particular, the discussion of
  \cite[Section~3.8]{janda2018frobenius} applies to the CohFT
  $\Omega^{\PP^1}$ of the degree zero $\CC^*$-equivariant
  Gromov--Witten theory of $\PP^1$.
  We first analyze this CohFT, then relate it to the Hodge CohFT, and
  finally study their tautological relations.

  The state space of $\Omega^{\PP^1}$ is the equivariant cohomology
  \begin{equation*}
    H^*_{\CC^*}(\PP^1) \cong \CC[\mu][H]/(H(H - \mu)) \cong \CC[\mu]\langle[0], [\infty]\rangle
  \end{equation*}
  viewed as a $2$-dimensional free module over the base ring $\CC[\mu]$.
  Here, $[0] = H$ and $[\infty] = H - \mu$ are the cohomology classes of
  the points $0$ and $\infty$ in $\PP^1$, respectively.
  Then $e_0 = [0]/\mu_0$ and $e_\infty = [\infty]/\mu_\infty$ for
  $\mu_0 = \mu$ and $\mu_\infty = -\mu$ form a basis of orthogonal
  idempotents.
  In particular, this CohFT is semisimple.
  In the basis $e_0$, $e_\infty$, the pairing is given by the matrix
  \begin{equation*}
    \begin{pmatrix}
      1/\mu_0 & 0 \\
      0 & 1/\mu_\infty
    \end{pmatrix}.
  \end{equation*}
  Note that there is the formula
  \begin{equation*}
    \Omega^{\PP^1}_{g, n}(\phi_1, \dotsc, \phi_n)
    = \sum_{i = 0}^g \lambda_i \cdot\int_{\PP^1} \prod_j \phi_j \cdot ([0] + [\infty])^{g - i}
  \end{equation*}
  for $\phi_1, \dotsc, \phi_n \in H^*_{\CC^*}(\PP^1)$.
  If we decompose $\phi_i = \phi_{i0} e_0 + \phi_{i\infty} e_\infty$,
  we may rewrite this as
  \begin{equation*}
    \Omega^{\PP^1}_{g, n}(\phi_1, \dotsc, \phi_n)
    = \sum_{\alpha \in \{0, \infty\}} \mu_\alpha^{g - 1} \prod_j \phi_{j\alpha} \cdot \sum_{i = 0}^g \lambda_i \mu_\alpha^{-i}
    = \sum_{\alpha \in \{0, \infty\}} \mu_\alpha^{g - 1} \Omega^{\lambda(\mu_\alpha^{-1})}_{g, n}(\phi_{1\alpha}, \dots, \phi_{n\alpha}).
  \end{equation*}
  Thus, $\Omega^{\PP^1}$ is a direct sum of two Hodge CohFTs.

  There is a Givental--Teleman reconstruction
  \begin{equation}
    \label{eq:Hodge-GT}
    \Omega^{\lambda(u)} = R . \omega,
  \end{equation}
  where
  \begin{equation*}
    R(z) = \exp\left(\sum_{i=1}^{\infty}\frac{B_{2i}}{(2i)(2i-1)}u^{2i-1}z^{2i-1}\right)
  \end{equation*}
  and $\omega_{g, n}(a_1, \dotsc, a_n) = a_1 \cdot \dotsb \cdot a_n$
  for $a_1, \dotsc, a_n \in \CC$.
  For $\Omega^{\PP^1}$, the Givental--Teleman reconstruction
  \begin{equation}
    \label{eq:P1-GT}
    \Omega^{\PP^1}_{g, n}(\phi_1, \dotsc, \phi_n) = R^{\PP^1} . \omega^{\PP^1}_{g, n}(\phi_1, \dotsc, \phi_n)
  \end{equation}
  respects the direct sum structure, with $R$-matrix
  \begin{equation*}
    R^{\PP^1}(z) =
    \begin{pmatrix}
      R(z)|_{u = \mu_0^{-1}} & 0 \\
      0 & R(z)|_{u = \mu_\infty^{-1}},
    \end{pmatrix}
  \end{equation*}
  if written in the basis $\{e_0, e_\infty\}$, and topological field
  theory defined via
  \begin{equation*}
    \omega^{\PP^1}_{g, n}(\phi_1, \dotsc, \phi_n) = \sum_{\alpha \in \{0, \infty\}} \mu_\alpha^{g - 1} \omega_{g, n}(\phi_{1\alpha}, \dotsc, \phi_{n\alpha})
  \end{equation*}

  The results of \cite[Section~3.8]{janda2018frobenius} imply that for
  any choice of $\phi_1, \dotsc, \phi_n$ (well-defined in the
  non-equivariant limit $\mu \to 0$), the polar terms in $\mu$ on the
  right hand side of \eqref{eq:P1-GT} are a linear combination of
  Pixton's 3-spin relations.
  If $n > 0$, note that because of the diagonal form of the pairing
  and $R^{\PP^1}$, we have
  \begin{equation*}
    \Omega^{\PP^1}_{g, n}([0], 1, 1, \dotsc, 1)
    = \mu^g \left((R . \omega)(1,\dots,1)\right)|_{u = \mu^{-1}},
  \end{equation*}
  and thus the polar terms correspond exactly to the terms in
  \eqref{eq:Hodge-GT} of degree $i > g$.
  This establishes the proposition in the case $n > 0$.
  
  The case $n = 0$ follows from the case $n = 1$ via the dilaton
  equation, and the fact that Pixton's relations are stable under
  pushing-forward via forgetful morphisms.
\end{proof}
\begin{remark}
  The result of Proposition~\ref{prop:3-spin} immediately implies that
  the differential equations for ancestor Gromov--Witten invariants
  from the vanishing $\lambda_i = 0$ for $i > g$ are linear
  combinations of those for Pixton's 3-spin relations.
  We expect a comptability between the Givental--Teleman
  reconstruction, quantization of quadratic Hamiltonian, and the
  descendant/ancestor correspondence, however we were not able to find
  a reference for this fact.
  This compatibility would imply that in a strong sense, the $i > g$
  part of \eqref{eqn:Hodge-class-relation} is a consequence of
  Pixton's 3-spin relations.
\end{remark}
\begin{remark}
  Given that Pixton's 3-spin relations are expected to be a complete
  set of relations in the tautological ring, we also expect the $i=g$
  part of \eqref{eqn:Hodge-class-relation} to also be a consequence of
  Pixton's 3-spin relations.
  However, we do not know how to prove this.
  One difficulty is that the two sides of
  \eqref{eqn:Hodge-class-relation} have very different origins.
  On a basic level, the left hand side of
  \eqref{eqn:Hodge-class-relation} comes from the
  Grothendieck--Riemann--Roch computation of the $\lambda_g$-class,
  while the right hand side comes from relative Gromov--Witten theory.
\end{remark}

~\;

\noindent {\bf  Acknowledgements.}
The authors would like to thank the anonymous referees for careful reading of the manuscript and for the many
helpful suggestions and comments.
The first author was partially supported by NSF grants DMS-2054830 and DMS-2239320.
The second author was partially supported by  National Key R\&D Program of
China (2022YFA1006200), National
Science Foundation of China (Grant No. 12071255) and Shandong Provincial Natural Science Foundation
(Grant No. ZR2021MA101). 
\hspace*{\fill}

\noindent{\bf Declarations}

\hspace*{\fill}

\noindent{\bf Conflict of interest} The authors declared that they have no conflicts of interest to this work.

\bibliographystyle{alpha}
\bibliography{references}

\end{document}